\documentclass[11pt]{amsart}

\usepackage[english]{babel}
 \usepackage{enumerate}
 \usepackage{xcolor}
\usepackage{amscd}
\usepackage{amsthm}
\usepackage[centertags]{amsmath}

\usepackage{newlfont}
\usepackage{graphicx}
\usepackage{amsfonts, amssymb,accents}
\usepackage{mathrsfs}
\usepackage{latexsym}
\usepackage{bbold}
\usepackage{enumitem}
  \definecolor{colorcita}{RGB}{21,86,130}
  \definecolor{colorref}{RGB}{5,10,177}
  \definecolor{colorweb}{RGB}{177,6,38}
 \usepackage[colorlinks=true,linkcolor=colorref,citecolor=colorcita,urlcolor=colorweb]{hyperref}

\newtheorem{theorem}{Theorem}[section]

\newtheorem{proposition}[theorem]{Proposition}
\newtheorem{corollary}[theorem]{Corollary}
\newtheorem{lemma}[theorem]{Lemma}
\newtheorem{remark}[theorem]{Remark}
\newtheorem{definition}[theorem]{Definition}

\newcommand{\zN}{\mathbb N}
\newcommand{\zK}{\mathbb K}

\newcommand{\sub}{\subseteq}

\newcommand{\F}{\mathcal {F}}

\title{Frequently recurrent backward shifts}

\author{Rodrigo Cardeccia, Santiago Muro}
\address{Instituto Balseiro, Universidad Nacional de Cuyo – C.N.E.A. and CONICET, San Carlos de Bariloche, Rep\'ublica Argentina} \email{rodrigo.cardeccia@ib.edu.ar} 
\address{FCEIA, UNIVERSIDAD NACIONAL DE ROSARIO AND   CIFASIS, CONICET}
\email{muro@cifasis-conicet.gov.ar}

\thanks{Partially supported by ANPCyT PICT 2015-2224, UBACyT 20020130300052BA, PIP 11220130100329CO and CONICET}

\begin{document}
\begin{abstract}
 We study  frequently recurrent backward shift operators on Fréchet sequence spaces. We prove that if a backward shift admits a nonzero frequently recurrent vector, then it supports a dense set of such vectors, and hence the operator is frequently recurrent. As a consequence, we provide two different characterizations for frequently recurrent backward shift operators,
which also imply the dense lineability of the respective set of
frequently recurrent vectors for this class of operators.

    For the construction of frequently recurrent vectors it is important to consider some arithmetic structure on the sets of return times along with their positive lower density. To achieve this, we introduce the recurrent arithmetic thickening of a sequence of subsets of $\zN$. We show that they provide an extension, that takes into account lower density, of a well-known result of Furstenberg  determining that sets of return times contain $IP$-sets. 
\end{abstract}

\maketitle

\section*{Introduction}

The central notion in the dynamics of linear operators is that of hypercyclicity. A linear operator \( T \) on a Fr\'echet space $X$ is said to be hypercyclic if it possesses a dense orbit, i.e. the set \( \{T^n x : n \in \mathbb{N}\} \) is dense in the space for some \( x \in X \).
Furthermore, if an orbit visits each nonempty open set with visiting times that grow at most linearly (or equivalently, the sets of visiting times have positive lower density), then the operator is called frequently hypercyclic. Since its introduction in \cite{BayGri06}, frequent hypercyclicity has been a fundamental concept in the development of linear dynamics.

The class of weighted shift operators (unilateral and bilateral) constitutes one of the most fundamental classes of operators and serves as a rich source of examples and counterexamples in linear dynamics; see for example, \cite{BayGri07,BayRuz15,Bes16,BonGro18,ChaErnMen16,charpentier2019chaos}. The characterization of hypercyclic weighted shifts is well-known, and in \cite{BayRuz15}, a characterization of frequently hypercyclic weighted shift operators on \(\ell_p\) and \(c_0\) was presented (see also \cite{BonGro18,charpentier2019chaos} for extensions to more general Fréchet sequence spaces).

While recurrence is a classical research topic in  dynamical systems, only recently in \cite{CosManPar14}, a systematic study on recurrence for linear operators was initiated. This line of research has been intensive in recent years \cite{abakumov2024orbits,BonGroLopPer22JFA,CardeMur22,CardeMur22multipleScand,GriLop23,GriLop25,lopez2024recurrent,LopMen25}. A vector \( x \) is (frequently) recurrent if the set of visiting times to each neighborhood of \( x \) is infinite (and has positive lower density). An operator is (frequently) recurrent if it has a dense set of (frequently) recurrent vectors.

Returning to weighted shifts, a result due to Chan and Seceleanu \cite{ChaSec12} implies that on \(\ell_p(\mathbb{N})\) and \(\ell_p(\mathbb{Z})\) for \(1 \leq p < \infty\), every recurrent weighted shift is hypercyclic and moreover, the existence of a single nonzero recurrent vector is sufficient to establish hypercyclicity. In particular, we have a zero-one type law for the set of nonzero recurrent vectors of a weighted shift: it is either empty or dense in the space. These zero-one laws properties  were extended to general Fr\'echet sequence spaces \cite{bonilla2025zeroEdinburgh,he2018jf-class} and investigated in other contexts, like adjoint of multipliers on function spaces \cite{ChSe10,bonilla2025zeroEdinburgh},   $\Gamma$-supercyclicity \cite{abakumov2024orbits} or chain recurrence for weighted shifts \cite{lopez2025shifts}. It is worth noting that zero-one laws do not hold for weighted shifts on directed trees \cite{abakumov2024several} or composition operators \cite{ChSe10}. 


In this article, we prove a zero-one law for frequently recurrent backward shifts and provide several applications. We will need to construct new recurrent vectors from a given one.
A well known property on (non-linear) dynamical systems is that the set of visiting times of a recurrent vector possesses a rich arithmetic structure. Indeed, it was proved by Furstenberg (see \cite[Theorem 2.17]{Fur81}) that they always contain $IP$-sets. Recall that a subset  $A\subset\mathbb N$ is an $IP$-set if there are natural numbers $p_1\le p_2\le\dots$ such that $A$ consists of
the numbers $p_i$ together with all finite sums $p_{i_1}+\dots + p_{i_k}$ with
$i_1<\dots<i_k$. Thus, if we want to construct a recurrent vector by pasting an appropriate sequence of vectors, we must do it carefully so that we respect the arithmetic structure. When we study frequent recurrence, there is a new issue: this arithmetic structure should somehow care about the lower density of the sets involved. Indeed, an inspection of the proof of   \cite[Theorem 2.17]{Fur81}, shows that the $IP$-set constructed there does not need to have positive lower density, no matter the properties of the recurrent vector. 



To solve this problem  we introduce what we call the \emph{recurrent arithmetic thickening} ($RAT$ from now on) of a sequence $(A_p)_p$ of subsets of $\zN$, denoted by $RAT((A_p)_p)$.  We show that these sets appear naturally in the context of recurrence: the sets of return times of recurrent points (on any dynamical system) always contain a $RAT$-set. The recurrent arithmetic thickening is always an $IP$-set and preserves the positivity of the lower density, so it becomes crucial for our  constructions of frequently recurrent vectors. The introduction of $RAT$-families allows us to prove an extension of \cite[Theorem 2.17]{Fur81} to arbitrary families.

The main results of the article are as follows: if a backward shift supports a non-trivial frequently recurrent vector, then there must be a dense set of such vectors. Equivalently, the operator is frequently recurrent. Moreover, using the idea of the construction of the dense set of frequently recurrent vectors, and by an application of the zero-one law, we can provide two different characterizations for frequent recurrence: the first one is in the same vein as the characterization of frequently hypercyclic shifts obtained in \cite{BonGro18}, while the second one is related to the existence of orbits that frequently approximate ``big'' subsets of the space $X$. We summarize our main results in a single statement.

\begin{theorem}\label{main theorem}
Let $X$ be a Fr\'echet sequence space with unconditional basis $(e_n)_{n\in \mathbb N_0}$ (resp. $(e_n)_{n\in \mathbb Z}$), let $B$ be the respective backward shift operator on $X$ and assume that it is continuous. 
 The following are equivalent:
\begin{enumerate}
    \item $B$ supports a nonzero frequently recurrent vector.
       \item $B$ is frequently recurrent.
    \item There is a dense vector subspace in which every nonzero vector is frequently recurrent and cyclic for $B$.
     \item  There are $(\varepsilon_p)_{p\ge 0}\subset \mathbb R_+$ with $\varepsilon_p\to 0$, and sets of positive lower density $A_p$  
     such that: 
\begin{enumerate}
    \item [i)]for every $p\in\mathbb N_0$, $\sum_{n\in A_p}e_{n+p}$ converges; 
    \item [ii)]
    for every $q\in\mathbb N_0$ and every  $m\in A_q$,
    $$
    \quad\Big\|\sum_{\overset{N\in RAT((A_p)_p)}{N>m}} e_{N-m}\Big\|_q<\varepsilon_q, \qquad \Big(\textrm{resp. } \Big\|\sum_{\overset{N\in RAT((A_p)_p)}{N\neq m}} e_{N-m}\Big\|_q<\varepsilon_q\Big).
    $$
        \end{enumerate}
   \item There are a set $A\subseteq \mathbb N_0$ of positive lower density and $x\in X$ with $supp(x)\subseteq A$, such that $x$ is a frequently hypercyclic vector for $\mathfrak B_1(A):=\{y\in X: supp(y)\subseteq A \text{ and } |y_n|<1 \text { for every } n\geq 0\}$. 
    
\end{enumerate}
\end{theorem}
We refer to Section \ref{section RAT} for the definition $RAT((A_p)_p)$ and to Section \ref{subsection segunda aplicacion} for the definition of \emph{frequent hypercyclicity for $Y$}, where $Y \sub X$ is a given subset of $X$ (the latter was originally introduced in \cite{Gro19bilateral}).

The results presented in the above statement depend on each other: for the proof of the dense lineability result $(3)$  (which is proved in Theorem \ref{cyclic and frequently recurrent vector}) we have to use the zero-one law ($(1)\Rightarrow (2)$) together with  the  characterization $(5)$ (which in turn uses $(4)$). 

The structure of the paper is the following. In the first section we present some preliminaries.
In Section \ref{section unilateral} we prove the main theorem of the paper, Theorem \ref{main theorem}, in the unilateral case. We will first prove the zero-one law for the case $X=c_0$ in Theorem \ref{0-1 law frequent recurrence c0}.  Although this result is in fact a particular case of Theorem \ref{0-1 law frequent recurrence}, we have chosen to present it separately because its more elementary proof captures the main idea behind the construction of frequently recurrent vectors and motivates the introduction of the $RAT$. This construction  follows a ``self-similar'' process which, in the limit, converges to a frequently recurrent vector. 
We introduce the $RAT$ and  extend \cite[Theorem 2.17]{Fur81}   to arbitrary families in Section \ref{section RAT}. We also show how the $RAT$ naturally fits in the context of frequently recurrent operators.
This is applied to show the general case of the zero-one law, i.e. $(1)\Leftrightarrow (2)$ of Theorem \ref{main theorem}, in Section \ref{section 0-1 general case}.  A careful look at the construction of frequently recurrent vectors in the zero-one law, together with some properties of the recurrent arithmetic thickening will allow us to characterize frequently recurrent backward shift operators in Theorem \ref{Characterization frequent recurrence} ($ (4)$ in Theorem \ref{main theorem}). Applying this characterization together with the zero-one law 
we show, in Theorem \ref{thm: caract S_1(A)}, a second characterization for frequently recurrent backward shifts ($ (5)$ in Theorem \ref{main theorem}). In Theorem \ref{cyclic and frequently recurrent vector} we apply all the previous results to obtain dense lineability.
Since our characterization $(4)$ is weaker than the characterization of frequent hypercyclicity from \cite{BonGro18}, it is natural to consider two classes of operators that are formally weaker than frequently hypercyclic but stronger than frequently  recurrent.  In Corollary \ref{Teo: caract freq hyp on S1} and Proposition \ref{prop: condicion con el min} we study them and describe the dynamical notions that they determine.

Finally, in Section \ref{section:bilateral} we prove the  main Theorem \ref{main theorem} in the bilateral setting.

\section{Preliminaries and notation}\label{preliminaries}
Through the paper $X$ will denote a Fr\'echet sequence space over $\mathbb K=\mathbb R$ or $\mathbb C$, i.e. $X$ is a subspace of $\mathbb K^{\mathbb N_0}$ or of $\mathbb K^{\mathbb Z}$ such that the functionals $X\ni x=(x_n)_n\mapsto x_n$ are continuous. The topology of $X$ is defined by a sequence of seminorms $(\|\cdot\|_p)_{p\in \mathbb N_0}$. We always assume that for each $x\in X$, the mapping $p\mapsto \|x\|_p$ is non-decreasing. Equivalently, the topology of $X$ may be defined via an $F$-norm, $\|\cdot\|=\sum_{p=0}^\infty \frac{1}{2^{p+1}} \min\{1,\|\cdot\|_p\}$. In that case, $(X,\|\cdot\|)$ results a complete metric space with distance $d(x,y)=\|x-y\|.$ 
Sometimes it will convenient to refer to $[x]_n$ as the $n$-th coordinate of $x$, that is $x_n$. 

An unconditional basis in a Fr\'echet sequence space $X$ is a basis which satisfies that $(a_n\cdot x_n)_n\in X$ whenever $a\in \ell_\infty $ and $x=(x_n)_n\in X$. This implies that for every $p\ge 0$ there are $m_p\geq p$ and a constant $C_p\ge1$ such that 
\begin{equation}\label{incondicionalidad}
    \|ax\|_p\le C_p\|a\|_{\infty} \|x\|_{m_p}.
\end{equation} 
We will always assume that the $C_p$'s are increasing.

Recall that an operator $T$ on $X$ is said to be hypercyclic provided that it has a dense orbit. It is called recurrent if it has a dense set of recurrent vectors, that is, of vectors $x\in X$ such that the set of hitting times $N_T(x,U):=\{n\in\mathbb N: T^n(x)\in U\}$ is infinite for each neighbourhood $U$ of $x.$ 

A set $A\sub \mathbb N_0$ is said to have positive lower density provided that
$$\underline {dens}(A):=\liminf_{n\to \infty} \frac{\# A\cap [0,n]}{n+1}>0$$
and positive upper density if 
$$\overline {dens}(A):=\limsup_{n\to \infty} \frac{\# A\cap [0,n]}{n+1}>0.$$

The lower density is finitely invariant:  $\underline {dens}(A)=\underline {dens}(A\cap F)$ for every cofinite set $F$. The same holds for the upper density.

Given an operator $T$ on $X,$
a vector $x\in X$ is said to be \emph{frequently recurrent} provided that for every open set $U$ around $x$ the set of hitting times $N_T(x,U) = \{n \in \mathbb{N} : T^n(x) \in U\}$ has positive lower density. If the set of frequently recurrent vectors is dense in the space, then we will say that the operator is frequently recurrent. A vector $x\in X$ is called \emph{frequently hypercyclic } if 
$\underline {dens}(N_T(x,U))>0$ for every nonempty open set $U\sub X$. If $T$ has a frequently hypercyclic vector then $T$ is said to be frequently hypercyclic. We refer to the books \cite{BayMat09,GroPer11} for basic concepts on linear dynamics.

Although we will mostly work with frequently recurrent backward shift operators, it will be useful to have a characterization of frequent hypercyclicity at hand. The following theorem was proved by Bonilla and Grosse-Erdmann in \cite{BonGro18}.
\begin{theorem}[Bonilla, Grosse-Erdmann]\label{caracterizacion hiperciclicidad frecuente} Let $X$ be a Fr\'echet sequence space with unconditional basis $(e_n)_{n\geq 0}$. Then $B$ is frequently hypercyclic if and only if there are sets with positive lower density $A_p$ and a sequence of positive nubers $\varepsilon_p\to 0 $ such that
\begin{enumerate}
    \item For every $p\in\mathbb N$, $\sum_{n\in A_p} e_{n+p}\in X$ and
    \item for every $p,q$, $m\in A_q$ and $0\leq j\leq p$
    $$\|\sum_{n\in A_p,n>m} e_{n-m+j}\|_q<\min\{\varepsilon_p,\varepsilon_q\}.$$
    
\end{enumerate}
\end{theorem}
We will always assume that the backward shift $B:X\to X$ is well defined (and hence continuous). Although we will mostly work with unweighted backward shifts all the results can be translated to the  weighted setting via an easy conjugation argument, if $B_w:X\to X$ is a weighted backward shift then we consider $v_n=\prod_{l=1}^n w_l^{-1}$ and $X(v)=\{(x_n)_n:(x_nv_n)_n\in X\}$. It follows that $B:X(v)\to X(v)$ is a well defined backward shift. Moreover, $B:X(v)\to X(v)$ and $B_w:X\to X$ are conjugate via the isomorphism $\phi: X(v)\to X$, $(x_n)_n\to (x_nv_n)_n $. See \cite{GroPer11,bonilla2025zeroEdinburgh} for more details on the matter.

\subsection{A separation Lemma for sets with positive lower density}\label{seccion separation lemma}
We now present a Separation Lemma \ref{separation lemma} for sets with positive lower density, which is crucial for our proofs. The idea of considering sets with separation properties is not new in linear dynamics. For example, a key step to prove that the operators satisfying the frequent hypercyclicity criterion are indeed frequently hypercyclic is a separation lemma of the natural numbers \cite[Lemma 2.2]{BayGri06}. In \cite[Proposition 1]{Bes16} the authors extended this lemma to general families of the natural numbers by proving that if there exist an $\F$-hypercyclic operator, then the family $\F$ must satisfy the separation lemma \cite[Lemma 2.2]{BayGri06}. We finally mention that the proof of the characterization of frequently hypercyclic backward shift operators \cite{BonGro18} uses several times that some sets $A_p$ with positive lower density are separated, where the $A_p$ are sets of hitting times of a frequently hypercyclic vector. 
These ideas have in common that  the separation property obtained is a consequence of a dynamical property: if $x\in X$ is a frequently hypercyclic vector for an operator $T$, and $(y_p)_p\subset X$ is a sequence of pairwise different vectors, then there are $\varepsilon_p>0$ such that the sets $A_p:=\{n\in\mathbb N:\|T^n(x)-y_p\|<\varepsilon_p\}$ satisfy some separation property. However, it is not possible to follow the same strategy for frequent recurrence because in this case the orbit is only known to  approximate a single vector. We have to change the strategy and to look at this property as a consequence of a number theoretic result rather than a dynamical property. To achieve this we will use a result proved recently by Martin, Menet and Puig in \cite{MarMenPui22} (see also \cite{BayGriMatMen24}, where the same lemma is applied).
\begin{lemma}[Separation Lemma]\label{separation lemma}
Let $(A_p)_p$ be a sequence of sets with positive lower density and $(k_p)$ a sequence of natural numbers. Then there exist pairwise disjoint sets of positive lower density $(B_p)_p$, such that for every $p$, $B_p\sub A_p$ and such that for every $n\in B_p$, $m\in B_q$, $n\ne m$, we have that $|n-m|>k_p+k_q$.
\end{lemma}

\section{Frequently recurrent backward shifts: the unilateral case}\label{section unilateral}
In this section we prove the unilateral case of the main Theorem \ref{main theorem} which is an immediate consequence of Theorem \ref{0-1 law frequent recurrence} (the zero-one law), Theorem \ref{Characterization frequent recurrence} (First characterization of frequent recurrence) and Theorem \ref{thm: caract S_1(A)} (Second characterization of frequent  recurrence).
As an application we prove in Subsection \ref{subsection segunda aplicacion} that the set of frequently recurrent vectors of a frequently recurrent backward shift is dense lineable.
An important ingredient of the construction of frequently recurrent vectors is the recurrence arithmetic thickening, which is addressed in Subsection \ref{section RAT}.

\subsection{A zero-one law in ${c_0}$}\label{subsection primer caracterizacion}
We begin this section by proving a zero-one law for frequently recurrent backward shift operators: the existence of one nonzero frequently recurrent vector implies that there is a dense set of such vectors. Then, by studying the construction of this dense set of vectors  we can describe the necessary conditions to obtain frequent recurrence, and we show hence a characterization of frequently recurrent shifts in Theorem \ref{Characterization frequent recurrence}.

We begin with the case of $c_0$, which is  technically less demanding and which illustrates the main underlying ideas for treating the general case. Moreover, the space $c_0$ has long served as a rich source of counter-examples in linear dynamics, for instance in \cite{BayGri07} the authors showed that frequent hypercyclicity  does not imply chaos, topological mixing or the existence of invariant Gaussian measures, by providing suitable examples of backward shifts on $c_0$. 

We will consider the backward shift operator on the weighted $c_0(v)$ space. Here $v=(v_n)_{n\ge 0}$ is a sequence of positive weights and for $x\in c_0(v)$ the norm is given by $\|x\|_{c_0(v)}=\sup_n|v_nx_n|$. We will assume, as we do in the whole manuscript, that $B$ is bounded, which in this case is equivalent to the condition of $(v_n/v_{n+1})_n$ being  bounded.

Our goal is to construct plenty of recurrent vectors starting from a single one. The construction of (frequently) recurrent vectors is in some sense fractal: We start by an initial vector $x_0$, and then we will  paste infinitely many copies of $x_0$ to obtain a vector $x_1\sim x_0$ and  such that the orbit of $x_1$ approximates $x_0$ along some set $A_1$ of positive lower density. In the $k$-step we will paste infinitely many copies of the vector $x_{k-1}$ and the orbit of $x_k$ will approximate each $x_{k-1}$ along  $A_k$. The vector $x=\lim_{k\to \infty} x_k$ will be our candidate to be a frequently recurrent vector. 

\begin{theorem}[Zero-one law for frequently recurrent vectors for $c_0$]\label{0-1 law frequent recurrence c0}
Let $X=c_0(v)$. If $B$ supports a nontrivial frequently recurrent vector then $B$ is frequently recurrent.
\end{theorem}
\begin{proof}
It suffices to prove that for every $R\in\mathbb N_0$, and every $\varepsilon>0$ there is a frequently recurrent vector $y\in X$ such that $\|y-e_R\|_{c_0(v)}<\varepsilon$. Indeed, if $U$ is a nonempty open set, let $z=\sum_{j=0}^R z_je_j\in U$. 
If $y\in X$ is a frequently recurrent vector sufficiently close to $e_R$, the continuity of $B$, of the sum and of the product by scalars imply that the vector $\sum_{j=0}^R z_j B^j(y)$ is a frequently recurrent vector in $U$.

So,  let $x\neq 0$ be a frequently recurrent vector,  $R\in\mathbb N_0$ and $\varepsilon>0$. Without loss of generality, we may suppose that $x_R\neq 0$ (otherwise consider $B^n(x)$ which is frequently recurrent). Let $\delta\in(0,1)$ such that $|x_R|>\delta.$

Since $x$ is a recurrent vector,  there is for every $p\in\mathbb N_0$, a  $q_p>p+R$ such that $|x_{q_p}|>\delta$. By the continuity of the coordinate functionals, there are a decreasing sequence of positive numbers $(\varepsilon_p)_p$ and a sequence of natural numbers $(k_p)_p$ such that for each $p\in\mathbb N_0$, $\varepsilon_{p}<\frac{\varepsilon}{2^{p+1}}$,  $k_p>p+1$ and
\begin{align}\label{eq: beta_p (zero one freq) UNIL c_0}
 \begin{split}
  \|w\|_{c_0(v)}<\varepsilon_{p}\Longrightarrow  |w_{j}|<\frac{\delta}{4} ,  \text{ for every } 0\leq j\leq q_p, &\quad \text{and}\\
 \|x-\sum_{0\leq j\le k_p} x_j e_j\|_{{c_0(v)}}<\frac{\varepsilon}{2^{p+1}}\frac{\delta}{8}.&
 \end{split}
 \end{align}

We will now present sets $A_p$  that will allow us to construct a frequently recurrent vector near $e_R$. For each $p\in\mathbb N_0$, $A_p$ will be a set of positive lower  density such that  
\begin{equation}\label{eq: A_p zero one freq c_0}
A_p\sub N_B(x,U_p), \text{ where } U_p=\{y\in X:\|y-x\|_{c_0(v)}<\varepsilon_p\frac{\delta}{8}\}.
\end{equation}
Applying the Separation Lemma \ref{separation lemma}
we may assume that the $A_p$'s are pairwise disjoint and that for every $n\in A_p,m\in A_q$ with $n\neq m$ we have that
\begin{equation}\label{eq separation zero one freq c_0}
|n-m|>k_p+k_q+2R>p+1+q+1+2R.    
\end{equation}

Notice that for every $p\in\zN_0$ and every $0\leq j\leq p$ we have that   $\sum_{n\in A_p} e_{n+R+j}$ converges. Indeed, by \eqref{eq: beta_p (zero one freq) UNIL c_0} and \eqref{eq: A_p zero one freq c_0} we have that $|x_{n+q_p}-x_{q_p}|<\frac{\delta}{4}$ and hence $|x_{n+q_p}|>\frac{\delta}{2}$ whenever $n\in A_p$. By uncondionality this implies that $\sum_{n\in A_p} e_{n+q_p}$ converges and by applying the backward shift $q_p-(R+j)$ times it follows that $\sum_{n\in A_p} e_{n+R+j}$ converges.  By the finite invariance property of  the sets of positive lower density, we may assume that 
\begin{equation}\label{eq A_p empieza en p c_0}
    A_p\sub [p+R+1,\infty)
\end{equation} and that,  for every $0\le j\le p,$
\begin{equation}\label{eq: sumas de e_j norma chica (zero one freq) c_0}
 \|\sum_{n\in A_p} e_{n+R+j}\|_{c_0(v)}<\frac{\varepsilon}{2^{p+1}}.   
\end{equation}
We will now proceed to construct a frequently recurrent vector near $e_R$ by induction. The rough idea is to construct first a vector $y^1$ near $e_R$ whose orbit is frequently near $y^0:=e_R$, and by induction a sequence of vectors $(y^k)_k$, contained in a small neighbourhood of $e_R$, and such that the orbit of $y^k$ is frequently near $y^{k-1}$. Those vectors will be constructed so that they   form a convergent sequence to a vector $y$, which will be a frequently recurrent vector near $e_R.$ Each vector $y^k$ will consist on carefully pasting together shifted copies of a truncation of the vector $y^{k-1}$ (and hence of copies of $e_R$).
 
Set $N_0=M_0=0$ and let $y^1=e_R+\sum_{n\in A_0}
e_{n+R}=e_R+\sum_{n\in A_0}
e_{n+N_0+R}$.  The vector is well defined, because of \eqref{eq: sumas de e_j norma chica (zero one freq) c_0}. Notice also that
$$\|y^1-e_R\|_{c_0(v)}=\|\sum_{n\in A_0}
e_{n+R}\|_{c_0(v)}<\frac{\varepsilon}{2}.$$

For the second step, let $N_1\in\zN$ such that $N_1+R$ is the second nonzero coordinate of $y^1$, which in this step coincides with $\min A_0$. Let $M_1>N_1$  such that $ \frac{\delta}{8}\|B\|^{N_{1}}\varepsilon_{M_{1}}<\frac{\varepsilon_0}{2^{2}}$.

We will add to $y^1$ infinitely many copies of a truncation $P_{N_1+R}(y^1)$ of $y^1$, where $P_{N_1+R}$ denotes the projection into the first $N_1+R$ coordinates:

\begin{align*}
  y^2&=y^1+\sum_{n\in A_{M_1}}\sum_{j=0}^{N_1+R} y^1_j e_{n+j} =y^1+\sum_{n\in A_{M_1}}e_{n+N_0+R}+e_{n+N_1+R}  =e_R+\sum_{p=0}^{1}\sum_{n\in A_{M_p}}\sum_{j=0}^p e_{n+N_j+R}.
  \end{align*}

Observe that, since $|n-m|>k_p+k_q+2R$ whenever $n\in A_p,m\in A_q$ and $k_p>p$, the sums involved have disjoint support. Also, since we imposed in \eqref{eq A_p empieza en p c_0} that $n>p+R$ for every $n\in A_p$ and $M_1>N_1$ we have that the first two nonzero coordinates of $y^2$ are $R$ and $N_1+R$.  It follows by \eqref{eq: sumas de e_j norma chica (zero one freq) c_0} that the sums are well defined and that 
$$\|y^2-y^1\|_{c_0(v)}=\|\sum_{n\in A_{M_1}}\sum_{j=0}^1 e_{n+N_j+R}\|_{c_0(v)}\leq  \frac{\varepsilon}{2^{M_1+1}}<\frac{\varepsilon}{2^2}.$$

Let us briefly sketch the third step. Let $N_2\in\zN$ such that $N_2+R$ is the third nonzero coordinate of $y^2$ (In this case $N_2=\min\{n\in A_{M_0}\cup A_{M_1}\,:\, n>N_1\}.$ Let $M_2>N_2$ such that $ \frac{\delta}{8}\|B\|^{N_{2}}\varepsilon_{M_{2}}<\frac{\varepsilon_0}{2^{3}}$.


Let us define $y^3$, with three copies of $e_R$ starting at each element of $n\in A_{M_2}$:
\begin{align*}
y^3&=y^2+\sum_{n\in A_{M_2}}\sum_{j=0}^{N_2+R} y^2_j e_{n+j}    =y^2+\sum_{n\in A_{M_2}} \sum_{j=0}^2  e_{n+N_j+R} =e_R+\sum_{p=0}^{2}\sum_{n\in A_{M_p}}\sum_{j=0}^p e_{n+N_j+R}.
\end{align*}
Then, proceeding as before, $y^3$ will be well defined and will satisfy, by \eqref{eq: sumas de e_j norma chica (zero one freq) c_0}, that $\|y^3-y^2\|_{c_0(v)}<\frac{\varepsilon}{2^3}$ and that its first three nonzero coordinates are $R,N_1+R$ and $N_2+R$.

We point out that there is a subtle difference from the fourth step onwards. Again $N_3\in\zN$ is defined so that $N_3+R$ is the fourth nonzero coordinate of $y^3$, but in this case  $N_3$ is not necessarily $\min\{n\in A_{M_0}\cup A_{M_1}\cup A_{M_2}\,:\, n>N_2\}.$ This is because it may happen that $N_1+N_2+R$ is the fourth nonzero coordinate of $y^3$, and $N_1+N_2\notin A_{M_0}\cup A_{M_1}\cup A_{M_2}$. In any case, since the nonzero coordinates of $y^3$ are of the form $n+N_j+R$, with $n\in A_{M_p}$, $1\leq p\leq 2$ and $0\leq j\leq p$, we have that 
$$N_3=\min(N_2,\infty)\bigcap\left(\bigcup_{p=0}^2\bigcup _{j=0}^p A_{M_p}+N_j\right).$$
The definition of $y^4$ and $M_3$ is analogous.

Let us now proceed to the inductive step. Suppose that we have defined $N_1,\dots, N_{k-1}$, $M_1,\dots,M_{k-1}$ and $y^1,\dots,y^k$ in the same way and that for every  $1\leq j\le k,$
$$
\|y^{j}-y^{j-1}\|_{c_0(v)}<\frac{\varepsilon}{2^{j}}.
$$
Define $N_k$ such that $N_k+R$ is the $(k+1)$-th nonzero coordinate of $y^k$  and let $M_k>N_k$ such that 
\begin{equation}\label{eq: definicion M_k c_0}
    \frac{\delta}{8}\|B\|^{N_{k}} \varepsilon_{M_{k}}<\frac{\varepsilon_0}{2^{k+1}}.
\end{equation}



Let
\begin{align}\label{eq: def y^{k+1}}
y^{k+1}=y^k+\sum_{n\in A_{M_{k}}}\sum_{j=0}^{N_{k}+R} y^k_j e_{n+j} =y^k+\sum_{n\in A_{M_{k}}}\sum_{j=0}^{k}  e_{n+N_j+R}=e_R+\sum_{p=0}^{k}\sum_{n\in A_{M_p}}\sum_{j=0}^p e_{n+N_j+R}.    
\end{align}
Note that by \eqref{eq: sumas de e_j norma chica (zero one freq) c_0},
\begin{align*}
\|y^{k+1}-y^k\|_{c_0(v)} &=\|\sum_{n\in A_{M_{k}}}\sum_{j=0}^{k}e_{n+N_j+R}\|_{c_0(v)}
\leq \frac{\varepsilon}{2^{M_{k}+1}}<\frac{\varepsilon}{2^{k+1}}.    
\end{align*}
This shows that $(y^k)_k$ is a Cauchy sequence so that the limit vector exists and that

\begin{align*}
y:=\lim_{k\to \infty} y^k &=e_R+\sum_{p=0}^{\infty}\sum_{n\in A_{M_p}}\sum_{j=0}^pe_{n+N_j+R}=e_R+\sum_{p=0}^\infty\sum_{n\in A_{M_p}}\sum_{j=0}^{N_p+R}y_je_{n+j}
\end{align*}
is well defined and that $\|y-e_R\|_{c_0(v)}<\varepsilon.$
Note that by the construction of $y^k$ we have that
\begin{align*}
y=\sum_{j=0}^\infty e_{N_j+R}.
\end{align*}

It remains to prove that $y$ is a frequently recurrent vector. From now and until the end of the proof, we fix $q\in\zN$ and $m\in A_{M_{q}}$. Recall that by \eqref{eq separation zero one freq c_0}, since $k_p>p$ for every $p,$ we have that for every $p\in\mathbb N$ and $n\in A_{M_p}$, $|n-m|>M_p+M_{q}+2R>N_p+N_{q}+2R$. Thus, if $n<m$ and $n\in A_{M_{p}}$, then $n-m+N_j+N_q+R<0$ for every $0\le j\le p$. Hence,
$$B^m\left(\sum_{n\in A_{M_p},n<m}\sum_{j=0}^pe_{n+N_j+R}\right)=\sum_{n\in A_{M_p},n<m}\sum_{j=0}^pe_{n-m+N_j+R}=0.$$
Noting that $m$ belongs only to $A_{M_{q}}$, because the $A_{M_p}$'s are pairwise disjoint, we get that
\begin{align*}
B^m(y)-\sum_{j=0}^{N_{q}+R}y_je_j &=B^m(y)-\sum_{j=0}^{q}e_{N_j+R}\\
&=\sum_{p\neq q}\sum_{n\in A_{M_p},n>m}\sum_{j=0}^{p}e_{n-m+N_j+R}+\sum_{n\in A_{M_{q}},n\geq m}\sum_{j=0}^{q}e_{n-m+N_j+R}-\sum_{j=0}^qe_{N_j+R}\\
&=\sum_{p=0}^\infty \sum_{n\in A_{M_p},n>m}\sum_{j=0}^pe_{n-m+N_j+R}\\
\end{align*}
Thus,
\begin{equation}\label{eq norma frec zero one c_0}
 \|B^m(y)-\sum_{j=0}^{N_{q}+R}y_je_j \|_{c_0(v)}= \left\|\sum_{p=0}^\infty \sum_{n\in A_{M_p},n>m}\sum_{j=0}^pe_{n-m+N_j+R}\right\|_{c_0(v)}.
\end{equation}

We claim that for every $N_k$ we have that $\|B^{N_k}x-x\|_{c_0(v)}<\varepsilon_0-\frac{\varepsilon_0}{2^k}\leq\varepsilon_0$.
Indeed, for $N_0=0$ it is plain and for $N_1=\min A_0$ it follows because $A_0\subseteq \{n\in \zN: \|B^n(x)-x\|_{c_0(v)}< \frac{\varepsilon_0 \delta}{8}\}$ and $\delta<1$.

Let $k>1$. By the construction of $N_k$, we have that $N_k+R$ is the  $(k+1)$-th non zero coordinate of $$y^{k}=e_R+\sum_{p=0}^{k-1}\sum_{n\in A_{M_p}} \sum_{j=0}^p e_{n+N_j+R}.$$ 
By a quick inspection we have that $N_k$ must be a number of the form $n+N_j$, where $n\in A_{M_p}$ for some $0\leq p\leq {k-1}$ and some $0\leq j\leq p$. Hence by \eqref{eq: definicion M_k c_0} and the inductive hypothesis,
\begin{align}
\label{eq: A contenido en N(x,epsilon)} \|B^{N_k}x-x\|_{c_0(v)} & =\|B^{n+N_j}x-x\|_{c_0(v)}\leq \|B^n B^{N_j}x-B^{N_j}x\|_{c_0(v)}+\|B^{N_j}x-x\|_{c_0(v)}    \\
\nonumber &\le \|B\|^{N_j}\,\|B^{n}x-x\|_{c_0(v)}+ \varepsilon_0(1-\frac1{2^j}) < \|B\|^{N_j-N_p}\,\frac{\varepsilon_0}{2^{p+1}}+  \varepsilon_0(1-\frac1{2^{j}}) < \varepsilon_0(1-\frac1{2^k}).
\end{align}

This implies that for every $j\in \zN_0$, $\|B^{N_j}x-x\|_{{c_0(v)}}<\varepsilon_0$ and since $|x_R|>\delta$, we have by \eqref{eq: beta_p (zero one freq) UNIL c_0} that $|x_{N_j+R}|>\frac{\delta}{2}$. On the other hand, we have that for every $n\in A_{M_p},$ $\|B^{n}x-x\|_{{c_0(v)}}<\varepsilon_{M_p}$ and hence by \eqref{eq: beta_p (zero one freq) UNIL c_0} and the fact that $q_{M_p}>M_p+R>N_j+R,$ we obtain for $0\leq j\leq p$ and $n\in A_{M_p}$ that
$|x_{n+N_j+R}-x_{N_j+R}|<\frac{\delta}{4}$. This implies that for every $n\in A_{M_p}$ and $0\leq j\leq p$, $|x_{n+N_j+R}|>\frac{\delta}{4}.$
 
 Since $m\in A_{M_q}$ and by \eqref{eq separation zero one freq c_0}, $n-m>M_p+k_{M_{q}}+R>k_{N_{q}}$ whenever $n\in A_{M_p}$ and $n>m$,  we have  $[\sum_{l=0}^{k_{N_{q}}}x_le_l]_{n-m+N_j+R}=0$ for each $0\le j\le p$. Hence it follows that  
  \begin{align*}
 \begin{split}
 \Bigg\|\sum_{p=0}^\infty \sum_{n\in A_{M_p},n>m}\sum_{j=0}^p & 
  e_{n-m+N_j+R} \Bigg\|_{c_0(v)}
 =\Bigg\|\sum_{p=0}^{\infty}\sum_{n\in A_{M_p},n>m}\sum_{j=0}^p\frac{1}{x_{n+N_j+R}}\cdot{x_{n+N_j+R}}e_{n-m+N_j+R}\Bigg\|_{c_0(v)}\\
 &=\Bigg\|\sum_{p=0}^{\infty}\sum_{n\in A_{M_p},n>m}\sum_{j=0}^p\frac{1}{x_{n+N_j+R}}\cdot[B^m(x)-\sum_{l=0}^{k_{N_{q}}}x_le_l]_{n-m+N_j+R}e_{n-m+N_j+R}\Bigg\|_{c_0(v)}\\
 &\leq \frac{4}{\delta}\left(\|B^m(x)-x\|_{{c_0(v)}}+\|x-\sum_{l=0}^{k_{N_{q}}}x_le_l\|_{{c_0(v)}}\right).
  \end{split}
 \end{align*}

Therefore, using  \eqref{eq norma frec zero one c_0}  together with \eqref{eq: beta_p (zero one freq) UNIL c_0} and the definition of $A_{M_p}$ in \eqref{eq: A_p zero one freq c_0}, we conclude that for any $m\in A_{M_{q}},$
\begin{align}\label{eq: B^my final c_0}
\|B^m(y)-\sum_{j=0}^{N_{q}+R}y_je_j \|_{c_0(v)}\le \frac{4}{\delta}\left(\|B^m(x)-x\|_{{c_0(v)}}+\|x-\sum_{l=0}^{k_{N_{q}}}x_le_l\|_{{c_0(v)}}\right)< \frac{\varepsilon}{2^{N_{q}+1}}.
    \end{align}

Finally let $V$ be an open set around $y$. Then since  ${N_q}\to\infty$ as $q\to\infty,$    $\sum_{j=0}^{N_q+R} y_je_j\to y$ and hence,  for big enough  $q\in\mathbb N$, it follows by \eqref{eq: B^my final c_0} that $A_{M_{q}}\sub N_B(y,V).$
This shows that $y$ is a frequently recurrent vector.
\end{proof}

\subsection{The recurrent arithmetic thickening\label{section RAT}}
A careful look at the proof of Theorem \ref{0-1 law frequent recurrence c0} shows that the sequences $(N_k)_k$ and $(A_{M_p})_p$ constructed there have an arithmetical relation, specifically, 
$$
N_{k}=\min\{ n>N_{k-1}\,:\, n\in\bigcup_{p=0}^{k-1}\bigcup_{j=0}^p A_{M_p}+N_j\}.
$$
This can be seen from equation \eqref{eq: def y^{k+1}} and recalling that $N_{k}$ is defined as the $(k+1)$-th nonzero coordinate of $y^k$. 
Moreover, if $A:=\{N_k:k\in\zN\}=\bigcup_{p=0}^{\infty}\bigcup_{j=0}^p A_{M_p}+N_j$, $A$ is contained in the set of return times of a small neighborhood of the recurrent vector, see \eqref{eq: A contenido en N(x,epsilon)}.

This at first sight strange condition should not be entirely surprising, since it is well-known that the sets of return times of recurrent vectors always contain some arithmetical structure. Indeed, it is shown in \cite[Theorem 2.17]{Fur81} that the sets of return times of recurrent points always contain $IP$-sets. A key point to achieve our goals is not only to obtain arithmetic structure, but to keep the  positive density of the set. 

Inspired by these comments we will define the recurrent arithmetic thickening of a sequence of subsets of $\zN$. 
\begin{definition}
    Given a sequence of subsets $A_p\sub \zN$ we define the recurrent arithmetic thickening of $(A_p)_p$ as the set
    \begin{equation}\label{RAT(A_p)}
    RAT((A_p)_p):=\{0\}\cup\bigcup_{p=0}^\infty\bigcup_{j=0}^p A_{p}+N_j ,\, \text{ where, }N_k=\min\{ n>N_{k-1}\,:\, n\in\bigcup_{p=0}^{k-1}\bigcup_{j=0}^p A_{p}+N_j\}
       \end{equation}
     for $k\ge1$ and $N_0=0$.
     \end{definition}
It is clear that, for any sequence $(A_p)_p$, it holds $(N_j)_j\sub RAT((A_p)_p)$. We will say that the sequence $(A_p)_p$ is \emph{compatible} if $RAT((A_p)_p)=(N_j)_j$.



\begin{remark}\label{rem:A_0 contenido en RAT}
    For any sequence of subsets $(A_p)_p$, let $N_k$ be the numbers appearing in the definition of $RAT((A_p)_p)$, i.e. $N_{k+1}=\min\{ n>N_{k}\,:\, n\in\bigcup_{p=0}^{k}\bigcup_{j=0}^p A_{p}+N_j\}$. Then $A_0\sub (N_k)_k$ and moreover, for any $0\le l\le p$,
    $$
    A_p\cap [N_p,\infty)+N_l\sub (N_k)_k.
    $$
    In particular, $(A_p)_p$ is compatible if $A_k\cap [1,N_k)=\emptyset$ for every $k\ge1.$
\end{remark}
\begin{proof}
Let $n\in A_p$, $n\ge N_p$. If $n=N_p$ the conclusion is obvius. If $n>N_p$ then $n+N_l\in A_p+N_l\subseteq  \bigcup_{m=0}^{k}\bigcup_{j=0}^m A_{m}+N_j$ for every $k\ge p$.  Then, since $N_k$ is the increasing sequence of minimums $N_{k+1}=\min\{ n>N_{k}\,:\, n\in\bigcup_{m=0}^{k}\bigcup_{j=0}^m A_{m}+N_j\}$ and $n+N_l> N_p$, $n+N_l$ will eventually coincide with some $N_{k}$, $k>p$.  Hence $n+N_l\in \{N_{p+1},N_{p+2},\dots\}$. 
\end{proof}

We will next show that the recurrent arithmetic thickening appears naturally for any recurrent point of any dynamical system and we will then apply this fact to show the general case of the Zero-one law.
\begin{proposition}\label{prop: condic necesaria vectores recurrentes}
    Let $(X,T)$ be a dynamical system (i.e. $X$ a first-countable topological space and $T:X\to X$  continuous). Let $x\in X$ be a recurrent vector, $(U_p)_p$ be a decreasing base of neighbourhoods of $x$ and $A_p\subset N_T(x,U_p)$ be non-empty sets. Then for each $q\in \zN$ there exists a subsequence $(A_{M_p})_p$ such that 
       $RAT((A_{M_p})_p)\subseteq N_T(x,U_q)$. 
\end{proposition}

\begin{proof}
Let $q\in \zN$. Let $M_0=q$. For $N_1=\min A_q$ we have $T^{N_1}x\in U_q$. Thus there is some $M_1$ such that $T^{N_1}(U_{M_1})\sub U_q$. Observe that this implies that for every $n\in A_{M_1}$ and every $0\leq j\leq 1$ we have that $T^{n+N_j}(x)\in U_q$. Indeed, for $j=0$ it is plain and for $j=1$, $T^{n+N_1}(x)=T^{N_1} T^n(x)\in T^{N_1}(U_{M_1})\subseteq U_q.$ 

We now define $M_{k+1}$ inductively. So we assume that we have defined $M_0,\dots,M_k$ such that for every $0\le j\le p\le k,$
$$
T^{N_j}(U_{M_p})\sub U_q,
$$
where the $N_j$ are defined by the relations \eqref{RAT(A_p)}.
Then since $N_{k+1}=n+N_j$ for some $n\in A_{M_p}$ with $0\le j\le p\le k,$ we have by the inductive hypothesis that $T^{N_{k+1}}x\in U_q$. So by continuity there is $M_{k+1}$ such that $T^{N_{j}}(U_{M_{k+1}})\in U_q$ for every $0\le j\le k+1$. 

It remains to show that $RAT((A_{M_p})_p)\subseteq N_T(x,U_q)$. Let $N\in RAT((A_{M_p})_p)$. Then  $N=n+N_j$ for some $n\in A_{M_k}$ with $0\le j\le k.$ Therefore $T^Nx=T^{n+N_j}x\in T^{N_j}(U_{M_k})\sub U_q $ by the definition of $M_k$.
\end{proof}
\begin{remark}\label{rem:M_k>N_k}
    Note that in the above proof $M_k$ is chosen after defining $N_k$, so we may assume without loss of generality that $M_k>N_k$.
\end{remark}

    We note next that for any sequence $(A_p)_p$ of infinite sets, the recurrent arithmetic thickening $RAT((A_{p})_p)$ is an $IP$-set. In particular we obtain as a consequence of \cite[Theorem 2.17]{Fur81} the following converse of Proposition \ref{prop: condic necesaria vectores recurrentes}.
\begin{proposition}\label{prop: los RAT son IP}
    Let $(A_p)_p$ any sequence of infinite sets, then the recurrent arithmetic thickening $RAT((A_{p})_p)$ contains an $IP$-set.  Moreover, there are a recurrent vector $x$  of some dynamical system $(X,T)$ and a neighborhood $U$ of $x$ such that $N_T(x,U)\sub RAT((A_{p})_p)$.

    Conversely, any $IP$-set contains $RAT((A_{p})_p)$ for some sequence  $(A_p)_p$ of infinite sets.  
\end{proposition}
\begin{proof}
    Take $m_1=N_1=\min A_0$ and $j_1=1$. Take $m_2\in A_{N_{j_1}}\cap (N_{j_1},\infty)$, where $(N_k)_k$ is the sequence in \eqref{RAT(A_p)}.
    By Remark \ref{rem:A_0 contenido en RAT}, we have that $\{m_1,m_{2}, m_1+m_2\}\sub \{N_k:k\in \zN\}$, thus     $m_1+m_2=N_{j_{2}}$ for some $j_{2}> j_1$. Then if $m_3\in A_{N_{j_2}}\cap (N_{j_2},\infty)$ we obtain again by Remark \ref{rem:A_0 contenido en RAT} that $\{m_1,m_{2},m_3, m_1+m_2,m_2+m_3,m_1+m_3,m_1+m_2+m_3\}\sub \{N_k:k\in \zN\}$. Proceeding inductively we obtain the first claim. 

    The second assertion now follows immediately from \cite[Theorem 2.17]{Fur81}.

    For the converse, if $A$ is the $IP$-set generated by the finite sums of an infinite set $A_0$. Consider the constant sequence given by $A_p=A_0$ for every $p$. Then $RAT((A_{p})_p)\sub A.$
\end{proof}

As mentioned above one advantage of the $RAT$ construction (over the $IP$-set construction) is that it allows to preserve properties like density of the original sets. This invites us to consider the notion of $RAT$-set associated to a given family in $\zN$.
\begin{definition}
    Let $\mathcal F$ be any hereditarily upward family of subsets of $\zN$ (i.e. $B_1\in\mathcal F$ and $B_1 \sub B_2$ imply $B_2\in\mathcal F$). We will say that a set $A\sub\zN$ is in $RAT_{\mathcal F}$ if there is a sequence $(A_p)_p$ of sets in $\mathcal F$ such that $ RAT((A_{p})_p)\sub A$. 
\end{definition}
Recall that, given an hereditary upward family $\mathcal F$, a vector $x$ is said to be \textit{$\mathcal F$-recurrent} for $T$ provided that for every open set $U$ around $x$ we have that $N_T(x,U)\in \mathcal F$.

Note that \cite[Theorem 2.17]{Fur81} implies that recurrence is equivalent to $IP$-recurrence.
As an immediate consequence of Proposition \ref{prop: condic necesaria vectores recurrentes} we obtain the following generalization of Furstenberg's Theorem to arbitrary families.
\begin{proposition}\label{prop: RAT_F}
Let $\mathcal F$ be an hereditary upward family and $(X,T)$ a dynamical system. Then $x\in X$ is a $\mathcal F$-recurrent point if and only if it is a $RAT_\mathcal F$-recurrent point.
\end{proposition}

We will not pursue any further on the investigation of $RAT_\mathcal F$-recurrence but we mention that  the notion of $RAT_\mathcal F$-recurrence, for $\mathcal F$ the family of sets with positive lower density, is of key importance for the proof of our main results. In the following construction we will check that
the vectors obtained are $RAT$-frequently recurrent, and hence frequently recurrent.

\subsection {Zero-one law. The general case}\label{section 0-1 general case}
We will prove now the zero-one law for arbitrary Fr\'echet spaces. The same construction of Theorem \ref{0-1 law frequent recurrence c0} works, however some technical details appear when dealing with unconditional basis on non-normable Fr\'echet spaces.
We will present a slightly different proof which relies on the recurrence arithmetic thickening: we first apply the separation Lemma to obtain separated subsets $A_p\subseteq N_B(x,U_p)$. Then we apply Proposition \ref{prop: condic necesaria vectores recurrentes} to obtain a compatible subsequence $A_{M_p}$ so that  $RAT((A_p)_p)=\{N_j\}$. The vector $1_{RAT((A_p)_p)}=\sum_{n\in {RAT((A_p)_p))}} e_{N_j}$ will be our candidate to be frequently recurrent. 


\begin{theorem}[Zero-one law for frequently recurrent vectors]\label{0-1 law frequent recurrence}
Let $X$ be a Fr\'echet sequence space with unconditional basis $(e_n)_{n\geq 0}$. If $B$ supports a nontrivial frequently recurrent vector then $B$ is frequently recurrent.
\end{theorem}

\begin{proof}
Let $x\neq 0$ be a frequently recurrent vector. It suffices to prove that for each $R\in\mathbb N_0$, and each $\varepsilon>0$ there is a frequently recurrent vector $y\in X$ such that $\|y-e_R\|<\varepsilon$, where $\|\cdot\|$ denotes an $F$-norm on $X$ (see Section \ref{preliminaries}). 

We may suppose that   $|x_R|>\delta$ for some $\delta>0$.   
Since $x$ is a recurrent vector,  there is for every $p\in\mathbb N_0$, a  $q_p>p+R$ such that $|x_{q_p}|>\delta$. By the continuity of the coordinate functionals,  there are a decreasing sequence of positive numbers $(\varepsilon_p)_p$ and a sequence of natural numbers $(k_p)_p$ such that for each $p\in\mathbb N_0$, $\varepsilon_{p}<\frac{\varepsilon}{2^{p+1}}$, $k_p>p+1$ and   
\begin{align}\label{eq: beta_p (zero one freq) UNIL}
 \begin{split}
  \|w\|_{k_p}<\varepsilon_{p}\Longrightarrow  |w_{j}|<\frac{\delta}{4} ,  \text{ for every } 0\leq j\leq q_p, &\quad \text{and}\\
 \|x-\sum_{0\leq j\le k_p} x_j e_j\|_{m_p}<\frac{\varepsilon}{2^{p+1}}\frac{\delta}{8C_p},&
 \end{split}
 \end{align}
 where $C_p$ and $\|\cdot\|_{m_p}$ are the unconditional constant basis and seminorm associated to $\|\cdot\|_p$ (see \eqref{incondicionalidad}).

We will now present sets $A_p\sub \zN$  that will allow us to construct a frequently recurrent vector near $e_R$. For each $p\in\mathbb N_0$, $A_p$ will be a set of positive lower  density such that  
\begin{equation}\label{eq: A_p zero one freq}
A_p\sub N_B(x,U_p),\quad\text{where }\, U_p=\left\{y\,:\, \|y-x\|_{k_p+m_p}<\varepsilon_p\frac{\delta}{8C_p}\right\}.
\end{equation}

Applying the Separation Lemma \ref{separation lemma}
we may assume that the $A_p$'s are pairwise disjoint and that for every $n\in A_p,m\in A_q$ with $n\neq m$ we have that
\begin{equation}\label{eq separation zero one freq}
|n-m|>k_p+k_q+2R>p+1+q+1+2R.    
\end{equation}

Notice that for each  $p\in \zN_0$ and each $0\leq j\leq p$ we have that, $\sum_{n\in A_p} e_{n+R+j}$ converges. Indeed, by \eqref{eq: beta_p (zero one freq) UNIL} and \eqref{eq: A_p zero one freq} we have that $|x_{n+q_p}-x_{q_p}|<\frac{\delta}{4}$ and hence $|x_{n+q_p}|>\frac{\delta}{2},
$ for all $n\in A_p.$ By unconditionality this implies that $\sum_{n\in A_p} e_{n+q_p}$ converges and by applying the backward shift $q_p-(R+j)$ times it follows that $\sum_{n\in A_p} e_{n+j+R}$ converges for every $0\leq j\leq p$. By the finite invariance property of  the sets of positive lower density, 
we may assume that 
\begin{equation}\label{eq A_p empieza en p}
    A_p\sub [p+R+1,\infty)
\end{equation}
and that,  for every $0\le j\le p,$
\begin{equation}\label{eq: sumas de e_j norma chica (zero one freq)}
 \|\sum_{n\in A_p} e_{n+R+j}\|<\frac{\varepsilon}{(p+1)2^{p+1}}.   
\end{equation}

Applying Proposition \ref{prop: condic necesaria vectores recurrentes} and Remark \ref{rem:M_k>N_k} there is a sequence $(M_p)_p$ for which $RAT((A_{M_p})_p)\subseteq N_B(x,U_0)$ and such that $M_p>N_p$ for every $p\geq 1$ and $M_0=N_0=0$. The fact that $M_p>N_p$ together with \eqref{eq A_p empieza en p}, implies that $(A_{M_p})_p$ is compatible, i.e. 
 $RAT((A_{M_p})_p)=\{N_k\}$.

We will show that the vector
\begin{align*}
y:=\hspace{-.3cm}\sum_{n\in RAT((A_{M_p})_p) }\hspace{-.3cm}e_{n+R}=e_R+\sum_{p=0}^{\infty}\sum_{n\in A_{M_p}}\sum_{j=0}^pe_{n+N_j+R}=e_R+\sum_{p=0}^\infty\sum_{n\in A_{M_p}}\sum_{j=0}^{N_p+R}y_je_{n+j}=\sum_{j=0}^\infty e_{N_j+R}.
\end{align*}
is a (well-defined) frequently recurrent vector such that $\|y-e_R\|<\varepsilon.$ The well-definition and the inequality $\|y-e_R\|<\varepsilon$ follow from \eqref{eq: sumas de e_j norma chica (zero one freq)}. 

We prove that $y$ is frequently recurrent. 
 Fix $q\in\zN$ and $m\in A_{M_{q}}$. We will show that $B^m(y)$ is close to $\sum_{j=0}^{N_{q}+R}y_je_j$.

 Recall that by \eqref{eq separation zero one freq}, since $k_p>p$ for every $p,$ we have that for every $p\in\mathbb N$ and $n\in A_{M_p}$, $|n-m|>M_p+M_{q}+2R>N_p+N_{q}+2R$. Thus, if $n<m$ and $n\in A_{M_{p}}$, then $n-m+N_j+N_q+R<0$ for every $0\le j\le p$. Hence,
$$B^m\left(\sum_{n\in A_{M_p},n<m}\sum_{j=0}^pe_{n+N_j+R}\right)=\sum_{n\in A_{M_p},n<m}\sum_{j=0}^pe_{n-m+N_j+R}=0.$$
Thus, noting that $m$ belongs only to $A_{M_{q}}$, because the $A_{M_p}$'s are pairwise disjoint, we get that
\begin{align*}
B^m(y)-\sum_{j=0}^{N_{q}+R}y_je_j &=B^m(y)-\sum_{j=0}^{q}e_{N_j+R}\\
&=\sum_{p\neq q}\sum_{n\in A_{M_p},n>m}\sum_{j=0}^{p}e_{n-m+N_j+R}+\sum_{n\in A_{M_{q}},n\geq m}\sum_{j=0}^{q}e_{n-m+N_j+R}-\sum_{j=0}^qe_{N_j+R}\\
&=\sum_{p=0}^\infty \sum_{n\in A_{M_p},n>m}\sum_{j=0}^pe_{n-m+N_j+R}.
\end{align*}
By Proposition \ref{prop: condic necesaria vectores recurrentes},  for $0\le j\le p,$ $n\in A_{M_p}$,
$$
B^{n+N_j}x\in U_0,\quad\text{and in particular,}\quad \| B^{n+N_j}x -x\|_{k_0+m_0}<\varepsilon_0.
$$
Then since $|x_R|>\delta,$ we have by \eqref{eq: beta_p (zero one freq) UNIL} that 
\begin{equation*}
    |x_{n+N_j+R}|\ge |x_R|-|x_{n+N_j+R}-x_R| \ge \delta - \frac{\delta}{4}.
\end{equation*}

Since by \eqref{eq separation zero one freq}, $n-m>k_{M_p}+k_{M_{q}}+R>k_{M_q}>k_{N_q}$ whenever $n\in A_{M_p}$ and $n>m$,  we have  $[\sum_{l=0}^{k_{N_{q}}}x_le_l]_{n-m+N_j+R}=0$ for each $0\le j\le p$. Hence it follows that  
  \begin{align}\label{eq: B^my final}
 \begin{split}
 \left\| B^m(y)-\sum_{j=0}^{N_{q}+R}y_je_j\right\|_q&=
 \left\|\sum_{p=0}^\infty \sum_{n\in A_{M_p},n>m}\sum_{j=0}^pe_{n-m+N_j+R}\right\|_q\\
 &=\left\|\sum_{p=0}^{\infty}\sum_{n\in A_{M_p},n>m}\sum_{j=0}^p\frac{1}{x_{n+N_j+R}}\cdot{x_{n+N_j+R}}e_{n-m+N_j+R}\right\|_q\\
 &=\left\|\sum_{p=0}^{\infty}\sum_{n\in A_{M_p},n>m}\sum_{j=0}^p\frac{1}{x_{n+N_j+R}}\cdot[B^m(x)-\sum_{l=0}^{k_{N_{q}}}x_le_l]_{n-m+N_j+R}e_{n-m+N_j+R}\right\|_q\\
 &\leq \frac{4C_q}{\delta}\left(\|B^m(x)-x\|_{m_q}+\|x-\sum_{l=0}^{k_{N_{q}}}x_le_l\|_{m_q}\right)\\
 & 
\le \frac{4C_q}{\delta}\left(\|B^m(x)-x\|_{k_{M_q}+m_{M_q}}+\|x-\sum_{l=0}^{k_{N_{q}}}x_le_l\|_{m_{N_q}}\right)< \frac{\varepsilon}{2^{N_{q}+1}},
  \end{split}
 \end{align}
where in the last inequality we have used \eqref{eq: beta_p (zero one freq) UNIL} and the definition of $A_{M_p}$ in \eqref{eq: A_p zero one freq}.

Finally let $V$ be an open set around $y$. Then since  ${N_q}\to\infty$ as $q\to\infty,$    $\sum_{j=0}^{N_q+R} y_je_j\to y$ and hence,  for big enough  $q\in\mathbb N$, it follows by \eqref{eq: B^my final} that $A_{M_{q}}\sub N_B(y,V).$
This shows that $y$ is a frequently recurrent vector.

\end{proof}

\subsection{First characterization of frequently recurrent backward shifts\label{section 0-1 general case, primera caracterizacion}}
A careful look to the proof of Theorem \ref{0-1 law frequent recurrence} allows us to isolate  the conditions that provide our first characterization for frequently recurrent backward shifts.

\begin{theorem}[First characterization of frequent recurrence]\label{Characterization frequent recurrence}
{Let $X$ be a Fr\'echet sequence space with unconditional basis $(e_n)_{n\geq 0}$. The following are equivalent: }
\begin{enumerate}
\item $B$ is frequently recurrent;
\item
 there are subsets $A_p$ with positive lower density 
 and a sequence $\varepsilon_p\to 0$,  such that
\begin{enumerate}
    \item [i)]for every $p\in\mathbb N_0$, $\sum_{n\in A_p}e_{n+p}$ converges; 
    \item [ii)]
    for any $q\in\mathbb N_0$ and every  $m\in A_q$,
    $$
    \left\|\sum_{\overset{N\in RAT((A_p)_p)}{N>m}} e_{N-m}\right\|_q<\varepsilon_q.
    $$
\end{enumerate}

\end{enumerate}
\end{theorem}  

\begin{proof}[Proof of $(1)\Rightarrow (2)$.] 
     This is a consequence of the proof of Theorem \ref{0-1 law frequent recurrence}. Indeed, we have seen there that if $B$ is frequently recurrent then taking $R=0$,  the sequences $(N_j)_j$, $(A_{M_q})_q$ and $(\varepsilon_p)_p=(\frac{\varepsilon}{2^{p+1}})_p$ satisfy condition $i)$ by \eqref{eq: sumas de e_j norma chica (zero one freq)} and condition $ii)$ by \eqref{eq: B^my final}.  
\end{proof}    

For the converse, we wish to assume that the sets $A_p$ are well separated. This is not longer an immediate application of separation Lemma \ref{separation lemma} because we don't know whether condition $ii)$ remains to hold when we take subsets of the $A_p's$.

To avoid these difficulties we will use the following lemma which after an application of the separation Lemma considers an appropriate subsequence $A_{M_p}'\sub A_{M_p}$ which form a compatible sequence. We will then show that the estimations $i)$ and  $ii)$ also hold for $RAT((A_{M_p}')_p)$.

{ 
\begin{lemma}\label{lem: RAT subconjuntos}
    Let $(A_p)_p$ be a sequence of subsets of $\zN$  and let  $A_p'\sub A_p\cap [N_p,\infty)$ be  a sequence of subsets. Then  there is a subsequence  $(M_p)_p$ (which can be chosen to grow as fast as desired)  such that 
    $(A_{M_p}')_p$ is a compatible sequence and $RAT((A_{M_p}')_p)\sub RAT((A_p)_p)$.
\end{lemma}

\begin{proof}
Let $N_1'=\min A_0'$. Since by Remark \ref{rem:A_0 contenido en RAT}, $A_0'\sub A_0\sub (N_k)_k$, there is $k_0\in \zN$ such that $N_1'=N_{k_0}$. 
Then $(A_{p_0}'+N_1')\cup A_{p_0}'\sub  (A_{p_0}+N_{k_0})\cup A_{p_0}\subseteq RAT((A_p)_p)$ for every $p_0\ge k_0$.
Thus if $M_1>\max\{k_0,N_1'\}$ then 
$$
\bigcup_{p=0}^{1} \bigcup_{j=0}^p A_{M_p}'+N_j'= A_0'\cup A_{M_1}'\cup (A_{M_1}'+N_1')\subseteq RAT((A_p)_p).
$$
Suppose now that we have defined $M_1,\ldots,M_{k-1}$, such that, 
\begin{enumerate}
    \item[a)] $N_{l}'=\min_{n>N_{l-1}'} \bigcup_{p=0}^{l-1} \bigcup_{j=0}^p A_{M_p}'+N_j' \in \{N_0,N_1,\dots,N_{M_l}\}$ for every $1\leq l\leq k-1$,
   \item[b)] $M_{l}>N_l'$ for every $1\leq l\leq k-1$,
    \item [c)]$\bigcup_{p=0}^{k-1} \bigcup_{j=0}^p A_{M_p}'+N_j'\subseteq RAT((A_p)_p)$.
\end{enumerate}

Let $N_k'=\min_{n>N_{k-1}'}\bigcup_{p=0}^{k-1} \bigcup_{j=0}^p A_{M_p}'+N_j'$. By Remark \ref{rem:A_0 contenido en RAT} and the fact that $A_p'\sub A_p\cap [N_p,\infty)$ it follows for $j\le p$ that since $N_j'<M_p<N_{M_p}$,  $A_{M_p}'+N_j' \in (N_i)_i$ for every $0\le j\le p\le k-1$. This implies that $N_k'\in (N_i)_i$.

Thus we have for large enough $p$, 
\begin{equation*}
    A_{p}'+N_j'\subseteq A_p+ N_j'\subseteq RAT((A_p)_p)
\end{equation*} 
for every $0\leq j\leq k$. So, we choose  $M_k$ large enough so that $A_{M_{k}}'+N_j'\subseteq RAT((A_p)_p)$ for every $0\leq j\leq k$, $M_k>N_k'$. Notice that, by c) we have  $\bigcup_{p=0}^{k} \bigcup_{j=0}^p A_{M_p}'+N_j'\subseteq RAT((A_p)_p)$.

By c) we obtain that $RAT((A_{M_p}')_p)\subseteq RAT((A_{p})_p)$.

Moreover, since $A_p'\sub [N_p,\infty)$ and  $M_p>N_p'$, Remark \ref{rem:A_0 contenido en RAT} implies that $(A_{M_p}')_p$ is compatible.
\end{proof}
\begin{remark}
    \label{rem: cosas del lemma que usamos}
We note for later use that the sequence $(M_p)_p$ in the above lemma can be chosen so that 
$$
N_{l}'\in \{N_0,N_1,\dots,N_{M_l}\}  \quad \text{ and, }\quad M_{l}>N_l' \text{ for every }l\ge1,
$$
\end{remark}
}

\begin{proof}[Proof of  $(2)\Rightarrow (1)$ in Theorem \ref{Characterization frequent recurrence}]    
 by Theorem \ref{0-1 law frequent recurrence} it suffices to construct a nonzero frequently recurrent vector near $e_0$. Let $(N_k)_k$ be the sequence associated to $RAT((A_p)_p)$, i.e. $N_k=\min \{n>N_{k-1}: n\in \bigcup_{p=0}^{k-1} \bigcup_{j=0}^p A_p+N_j\}$ and $N_0=0.$

We begin applying the separation Lemma \ref{separation lemma} to obtain sets of positive lower density $A_p'\subseteq A_p$  such that 
\begin{equation}\label{eq: caracterizacion 1 separation de los A_p}
    |n-m|>p+q
\end{equation} whenever $n\in A_p',m\in A_q'$ and $n\neq m$. By $i)$ and unconditionality,  $\sum_{n\in A_p'} e_{n+j}\in X$; and by finite invariance we may suppose that  $A_p'\sub [N_{p}+1,\infty)$ and that 
\begin{equation}\label{eq: caracterizacion 1 finitely invariance}
    \|\sum_{n\in A_p'} e_{n+j}\|\leq \frac{1}{2^p}
\end{equation} for every $0\le j\le p$.

By Lemma \ref{lem: RAT subconjuntos}, there is a subsequence $(M_p)_p$ such that $(A_{M_p}')_p$ is compatible and $RAT((A_{M_p}')_p)\sub RAT((A_p)_p)$.
 Moreover, if $RAT((A_{M_p}')_p)=(N_j')_j$, $N_{l}'=\min_{n>N_{l-1}'} \bigcup_{p=0}^{l-1} \bigcup_{j=0}^p A_{M_p}'+N_j'$ for every $l$ and we may assume that 
\begin{enumerate}
   \item[a)] \label{eq: caracterizacion 1 Los M_l son grandes} $M_{l}>\max\{m_{l},N_l'\}$ for every $1\leq l$, (see Remark \ref{rem: cosas del lemma que usamos})
   \item [b)] $\varepsilon_{M_l}<\varepsilon_{M_{l-1}}$ for $1\leq l$, and,
    \item [c)] $C_l\cdot\varepsilon_{M_l}<\frac{1}{l}$ for every $1\leq l$,
\end{enumerate}
where $C_l$ and $m_l$ are the constants arising     from the unconditionality of the basis.



We have so far obtained sets of positive lower density $A_{M_p}'$  satisfying the separation property \eqref{eq: caracterizacion 1 separation de los A_p} and the estimation \eqref{eq: caracterizacion 1 finitely invariance}. Moreover, since $RAT((A_{M_p}')_p)\subseteq RAT((A_{p})_p)$, we have by unconditionality  that $RAT((A_{M_p}')_p)$ satisfy the inequality in $ii)$ for  $\varepsilon_q'=\frac1{q}$
, i.e. for every $q\in \zN$ and $m\in A_{M_{q}}'$ we have that
\begin{equation}\label{eq: norma de la suma RAT(A') <1/q}
\left\|\sum_{\overset{N\in RAT((A_{M_p}')_p)}{N>m}} e_{N-m}\right\|_q<\frac{1}{q}.    
\end{equation}

Indeed, let $m\in A_{M_q}'$. We have by the separation property  \eqref{eq: caracterizacion 1 separation de los A_p}    together with b) that for $n\in A_{M_p}'$,
\begin{equation}\label{eq: separation Ap' first characterization}
|m-n|>M_{q}+M_p> M_{q}+N_p'.  
\end{equation}
Thus $\displaystyle \sum_{\overset{N\in RAT((A_{M_p}')_p)}{N>m}} e_{N-m} =    \sum_{p=0}^\infty\sum_{j=0}^p\sum_{\overset{n\in A_{M_{p}}'}{n+N_j'>m}} e_{n-m+N_j'}=    \sum_{p=0}^\infty\sum_{j=0}^p\sum_{\overset{n\in A_{M_{p}}'}{ n>m}} e_{n-m+N_j'}$. 

Moreover, by Remark \ref{rem: cosas del lemma que usamos} we have that for every $p\in \zN$, $\{N_1',\ldots ,N_{p}'\}\subseteq \{N_1,\ldots , N_{M_p}\}$  and hence, using unconditionality,  a) and c), we obtain that
\begin{align}
\nonumber \left\|\sum_{\overset{N\in RAT((A_{M_p}')_p)}{N>m}} e_{N-m}\right\|_q & =    \left\|\sum_{p=0}^\infty\sum_{\overset{n\in A_{M_{p}}'}{ n>m}}\sum_{j=0}^p e_{n-m+N_j'}\right\|_q \leq C_q\cdot \left\|\sum_{p=0}^\infty\sum_{\overset{n\in A_{M_p}}{n>m}}\sum_{j=0}^{M_p} e_{n-m+N_j}\right\|_{m_q}\\
 \label{eq: norma de la suma n>m <1/q}   & \leq C_q\cdot\left\|\sum_{p=0}^\infty\sum_{\overset{n\in A_{M_p}}{n>m}}\sum_{j=0}^{M_p} e_{n-m+N_j}\right\|_{M_q}
     <C_q\cdot \varepsilon_{M_{q}}
    <\frac{1}{q}. 
\end{align}

The hard work has already been done. Denoting $A'=RAT((A_{M_p}')_p)$, we have that $1_{A'}=\sum_{n\in A'} e_n=\sum_{j=0}^\infty e_{N_j'}=\sum_{p=0}^\infty\sum_{n\in A_{M_{p}}'}\sum_{j=0}^p e_{n+N_j'} $ converges because, by \eqref{eq: caracterizacion 1 finitely invariance}, $\sum_{p=0}^\infty \|\sum_{n\in A_{M_p}'} \sum_{0\leq j\leq p} e_{n+N_j'}\|\leq \sum_{p=0}^\infty \frac{p+1}{2^p}<\infty$. We show next that $1_{A'}$ is a frequently recurrent vector.

Let $m\in A_{M_{q}}'$. Then for $N\in A_{M_p}'$, $N<m$ we have  that $ B^m(e_{N})=0$. Hence, by \eqref{eq: norma de la suma n>m <1/q},
\begin{align*}
\left\|B^{m}(1_{A'})-\sum_{j=0}^{q} e_{N_j'}
\right\|_q & = \left\|\sum_{\overset{N\in A'}{N\ge m}} e_{N-m} -\sum_{j=0}^{q} e_{N_j'}\right\|_q \\
&= \left\|\left(\sum_{p=0}^\infty\sum_{\overset{n\in A_{M_{p}}'}{n>m}}\sum_{j=0}^p e_{n-m+N_j'}+\sum_{j=0}^{q} e_{N_j'}\right)-\sum_{j=0}^{q} e_{N_j'}\right\|_q
<\frac{1}{q},    
\end{align*}
which implies that $1_{A'}$ is frequently recurrent.
\end{proof}

With the same proof we can present another characterization which follows the spirit of the characterization of frequent hypercyclicity from Theorem \ref{caracterizacion hiperciclicidad frecuente} obtained by Bonilla and Grosse-Erdmann in \cite{BonGro18}.
If the sets $(A_p)_p$ are sufficiently separated then the equivalence $(2)$ in the following just rephrases Theorem \ref{Characterization frequent recurrence}. The equivalence $(3)$ is a consequence of the proof of Theorem \ref{Characterization frequent recurrence}. 
\begin{theorem}\label{coro: first Characterization frequent recurrence}
{Let $X$ be a Fr\'echet sequence space with unconditional basis $(e_n)_{n\geq 0}$. The following are equivalent: }
\begin{enumerate}
\item $B$ is frequently recurrent;

\item  there  are subsets $A_p\subseteq \zN$ with positive lower density, $(N_k)_k$ defined by the relation
 $N_k=\min_{n>N_{k-1}} \{n\in\bigcup_{p=0}^{k-1}\bigcup_{j=0}^p A_p+N_j\}$  
  and a sequence $\varepsilon_p\to 0$,  such that
\begin{enumerate}
    \item [i)]for every $p\in\mathbb N_0$, $\sum_{n\in A_p}e_{n+p}$ converges; 
    \item [ii)]
    for any $q\in\mathbb N_0$ and every  $m\in A_q$,
    $$
    \left\|\sum_{p=0}^\infty \sum_{n\in A_p, n>m} \sum_{j=0}^p e_{n-m+N_j}\right\|_q<\varepsilon_q.
    $$

\end{enumerate}

\item  there is a compatible sequence  $(A_p)_p$ formed by subsets $A_p\subseteq \zN$ with positive lower density such that
$$
\sum_{n\in RAT((A_p)_p)}\hspace{-.5cm}e_{n}\quad\text{is a frequently recurrent vector.}
$$
\end{enumerate}

 Moreover in (2) we may assume that the sequence $(A_p)_p$ satisfies:
\begin{itemize}
        \item it is well separated: for every $p,q\in\mathbb N_0$ and every $n\in A_p$, $m\in A_q$ with $n\neq m$ we have that $|n-m|>p+q.$
    \item for every $p\in \zN_0$, $\|\sum_{n\in A_p} e_{n+N_p}\|$ is arbitrarily small, and
    \item it is compatible: $\{N_j\}=RAT((A_p)_p).$
\end{itemize}

\end{theorem}

This characterization of frequent recurrence is similar to the characterization of frequent hypercyclicity from Theorem \ref{caracterizacion hiperciclicidad frecuente} obtained by Bonilla and Grosse-Erdmann in \cite{BonGro18}, but it is weaker in two aspects: while condition i) is the same on both characterizations, we require in ii) an upper bound which depends only on $\varepsilon_q$ (and not on the minimum between  $\varepsilon_p$ and $\varepsilon_q$). Also, the bound of ii) holds only for a subset of the natural numbers $A=\{N_k\}$, while in \cite{BonGro18} the authors need $A=\mathbb N_0$. Of course, the case $A=\mathbb N_0$ is a sufficient condition for frequent recurrence.
\begin{corollary}\label{Corolario zN}
Let $X$ be a Fr\'echet sequence space with unconditional basis $(e_n)_{n\in \mathbb N}$. Suppose that there are sets $A_p$ with positive lower density such that
\begin{enumerate}
\item [i)]for every $p\in\mathbb N_0$, $\sum_{n\in A_p}e_{n+p}$ converges 
    \item [ii)]
    For every $q\in\mathbb N_0$ and every  $m\in A_q$,
    $$\|\sum_{p=0}^\infty \sum_{n\in A_p, n>m} \sum_{j=0}^p e_{n-m+j}\|_q<\varepsilon_q.$$
    \end{enumerate}Then $B$ is frequently recurrent.
\end{corollary}
\begin{proof}
The proof is similar: we construct inductively a sequence $(M_p)_p$ such that $M_0=0$,
  $M_{l}>\max\{m_{l},N_l'\}$,
 $\varepsilon_{M_l}<\varepsilon_{M_{l-1}}$ and,
  $C_l\cdot\varepsilon_{M_l}<\frac{1}{l}$ for every $1\leq l$,
where $C_l,m_l$ are the constants from the unconditionality, $N_0'=0$ and $N_k' =\min_{n>N_{k-1}'} \{n\in\bigcup_{p=0}^{k-1}\bigcup_{j=0}^p A_{M_p}+N_j'\}$.

Hence we have,
\begin{align*}
\|\sum_{p=0}^\infty \sum_{n\in A_{M_p}, n>m}\sum_{j=0}^p e_{n-m+N_j}\|_q &\le C_q\|\sum_{p=0}^\infty \sum_{n\in A_{p}, n>m}\sum_{j=0}^p e_{n-m+j}\|_{m_q}\\
&\le C_q\|\sum_{p=0}^\infty \sum_{n\in A_{p}, n>m}\sum_{j=0}^p e_{n-m+j}\|_{M_q}<C_q\varepsilon_{M_q}<\frac1{q}.   
\end{align*}
Thus, the sequence $(A_{M_p})_p$ satisfies the conditions of Theorem \ref{coro: first Characterization frequent recurrence} (2) and hence $B$ is frequently recurrent.
\end{proof}

Weighted backwards shifts in $c_0$ provide a great source of counter-examples in linear dynamics \cite{BayRuz15,BonGro18,BayGri07,Bes16}. It is therefore meaningful to enunciate Theorem \ref{Characterization frequent recurrence} and Corollary \ref{Corolario zN} for the space $c_0$.
\begin{theorem}
Let $X=c_0$ and $B_w$ be a weighted backward shift. Then $B_w$ is frequently recurrent if and only if there are a sets of positive lower density $A_p$  such that 
\begin{enumerate}
    \item [i)] for each $p\in \zN_0$, $\lim_{n\in A_{p}} \prod_{l=1}^{n+p} w_{l}=0$; 
\end{enumerate}    
    and such that one of the following conditions holds
    \begin{enumerate}
        
    \item [ii)]for every $q\in \zN_0$, every $n\in RAT((A_p)_p)$ and every $m\in A_q$ with $n>m$ we have that
$$\prod_{l=1}^{n-m} w_{l}\leq \varepsilon_q$$
for every $0\leq j\leq p$;
    \item [ii')] If $N_k=\min_{n>N_{k-1}} \{\bigcup_{p=0}^{k-1}\bigcup_{j=0}^k A_p+N_j\}$, then for every $p,q\in \zN_0$ and every $n\in A_p$, $m\in A_q$ with $n>m$ we have that
    
$$\prod_{l=1}^{n-m+N_j} w_{l}\leq \varepsilon_q.$$
    \end{enumerate}
\end{theorem}
\begin{corollary}
Let $X=c_0$ and $B_w$ be a weighted backward shift. Suppose that there are sets of positive lower density $A_p$ and a sequence of positive numbers $\varepsilon_p\to 0$ such that 
\begin{itemize}
    \item[(i)] for every $p\in\zN_0$, $\lim_{n\in A_{p}} \prod_{l=1}^{n+p} w_{l}=0$
    \item[(ii)] for every $p,q\in \zN_0$ and every $n\in A_p$, $m\in A_q$ with $n>m$ we have that
$$\prod_{l=1}^{n-m+j} w_{l}\leq \varepsilon_q$$
for every $0\leq j\leq p$.
\end{itemize} Then $B$ is frequently recurrent.
\end{corollary}

\subsection{Second characterization of frequent recurrence and an application to dense lineability. }\label{subsection segunda aplicacion}
We will now show another equivalence for frequent recurrence in terms of the existence of  orbits with wild behavior. As an application we will show that the set of frequently recurrent vectors is dense lineable. Given a set $A\subseteq \zN_0$ we consider $\mathfrak B_M(A)$ the set of vectors supported in $A$ that have coordinates bounded by $M>0$, i.e.
$$\mathfrak B_M(A):=\{x\in X: supp(x)\subseteq A \text{ and } |x_n|<M \text { for every } n\geq 0\}.$$

We recall the following definition from \cite {Gro19bilateral}:  given $T:X\to X$ and a subset $Y\subseteq X$ we say that $x\in X$ is a \emph{frequently hypercyclic vector for $Y$} if for every open set $U$ such that $U\cap Y\neq \emptyset$ we have that $N_T(x,U)$ has positive lower density.

\begin{theorem}[Second characterization of frequent recurrence] \label{thm: caract S_1(A)}
Let $X$ be a Fr\'echet sequence space with unconditional basis $(e_n)_{n\ge 0}$. Then $B$ is frequently recurrent if and only if there are $A\subseteq \mathbb N_0$ with $\underline {dens}(A)>0$ and $x\in X$, supported in $A$, such that $x$ is a frequently hypercyclic vector for $\mathfrak B_1(A)$. 
\end{theorem}

\begin{proof}
If there are $A$ and $x\in X$ supported in $A$ such that $x$ is frequently hypercyclic for $\mathfrak B_1(A)$ then we have by unconditionality that $P_{\mathfrak B_1(A)} x:= \sum_{n\in A} \min\{|x_n|,1\}sgn(x_n) e_n$ is frequently hypercyclic for $\mathfrak B_1(A)$, where $sgn$ denotes either the real or complex sign of a number depending the field of $X$. Indeed, for any $y\in \mathfrak B_1(A)$, $q\in \mathbb N_0$, we have,
$$
\|B^m(P_{\mathfrak B_1(A)} x)-y\|_q\le C_q \|B^m( x)-y\|_{m_q}.
$$

Since $P_{\mathfrak B_1(A)} x$ belongs itself to $\mathfrak B_1(A)$ then it is a frequently recurrent vector. By the zero-one law Theorem \ref{0-1 law frequent recurrence} we deduce that the operator is frequently recurrent.  

Suppose now that $B$ is frequently recurrent. 
By Theorem \ref{coro: first Characterization frequent recurrence} there are $A_p$'s and $\varepsilon_p$'s, 
\begin{itemize}
    \item for every $p,q\in\mathbb N_0$ and every $n\in A_p$, $m\in A_q$ with $n\neq m$ we have that $|n-m|>p+q.$
    \item for every $p\in \zN_0$, $\|\sum_{n\in A_p} e_{n+N_p}\|\leq\frac{1}{(p+1)2^{p+1}}$ and
    \item $(A_p)_p$ is compatible.
\end{itemize}

Let $A=RAT((A_p)_p)=\{N_j\}$.
The above conditions  imply that the vector $z=\sum_{p=0}^\infty\sum_{n\in A_p} \sum_{j=0}^{p} e_{n+N_j}=\sum_{n\in A} e_n$ is in $X$. 

Let $(y^p)_p$ be a dense sequence in $\mathfrak B_1(A)$ such that $supp(y^p)=\{0,N_1,\ldots, N_p\}.$

Consider now the vector $y=\sum_{p=0}^\infty \sum_{n\in A_p} \sum_{j=0}^p y^p_j e_{n+N_j}.$ This vector is well defined, because 
\begin{align*}
\sum_{p=0}^\infty \|\sum_{n\in A_p} \sum_{j=0}^p y^p_j e_{n+N_j}\|&\leq \sum_{p=0}^\infty \frac{1}{ 2^{p+1}}<\infty.    
\end{align*}

We show now that $y$ is a frequently hypercyclic vector for $\mathfrak B_1(A)$. Let $q\in \zN_0$ and $m\in A_{m_q}$, then by condition $(2)$ $ii)$ of Theorem \ref{coro: first Characterization frequent recurrence} and unconditionality,
$$\|B^m(y)-y^{q}\|_q=\|\sum_{p=0}^\infty \sum_{n\in A_p,n>m}\sum_{j=0}^p y^p_j e_{n-m+N_j}\|_q\leq C_q\varepsilon_{m_q}.$$

\end{proof}
We finish the subsection with an application of Theorem \ref{thm: caract S_1(A)}. Recall that a set $F\subset X$ is said to be dense lineable provided that $F\cup\{0\}$ contains a subspace which is dense in $X$. In the context of linear dynamics, it is immediate that the set of (frequently) hypercyclic vectors is dense lineable, because if $x$ is a (frequently) hypercyclic vector then $\{P(T)(x):P \text{ is a polynomial}\}$ is a dense space of (frequently) hypercyclic vectors. However, the same argument breaks for recurrence. In \cite{LopMen25} the authors found a recurrent operator without a dense lineable set of recurrent vectors. It is still an open question whether the set of frequently recurrent vectors is dense lineable, see \cite[Problem 5.1 and Remark 5.3]{GriLopPerQuestions2} and \cite[Problem 4.2]{LopMen25}. We provide a positive answer for backward shift     operators (both for the unilateral and bilateral case). We note also that in \cite[Corollary 4.8]{GriLopPerQuestions2} it was shown that if $T$ supports a dense set of vectors such that $T^nx\to 0$ (in particular if $T$ is a unilateral backward shift) and if the family $\mathcal F$ is an upper right-invariant Furstenberg family $\mathcal F$, then the $\mathcal F$-recurrence of $T$ implies by \cite{BonGroLopPer22JFA} the $\mathcal F$-hypercyclicity of $T$ and in particular the dense lineability for the set of $\mathcal F$-recurrent vectors of $T$. 

\begin{theorem}\label{cyclic and frequently recurrent vector}
Let $X$ be a Fr\'echet sequence space with unconditional basis $(e_n)_{n\in\mathbb N_0}$. Then  $B$ is frequently recurrent if and only if $B$ has a cyclic and frequently recurrent vector.

In particular, if $B$ is frequently recurrent then
\begin{enumerate}
    \item the set of frequently recurrent vectors is dense lineable and moreover,
    \item the set of frequently recurrent vectors of $B\times\cdots \times B:X^n\to X^n$ is dense lineable for every $n\in\mathbb N.$
\end{enumerate}
\end{theorem}
\begin{proof}
If $B$ supports a cyclic and frequently recurrent vector, then the operator is frequently recurrent by the zero-one law Theorem \ref{0-1 law frequent recurrence}.

Conversely suppose that $B$ is frequently recurrent. Then there are, by  Theorem \ref{thm: caract S_1(A)}, a set $A$, with $\underline{dens}(A)>0$, and $x\in X$ supported in $A$ such that $x$ is a frequently hypercyclic vector for $\mathfrak B_1(A)$. By the proof of Theorem \ref{thm: caract S_1(A)} we can suppose that $x\in \mathfrak{B}_1(A)$. In particular, $x$ is frequently recurrent. We claim that $x$ is also a cyclic vector.

Let $y=\sum_{j=0}^N y_j e_j$ and $\varepsilon>0$. For each $n\in \mathbb N$ consider $A_n:=A-n$ and let $n_1,\ldots, n_r$ such that $[0,N]\subseteq \bigcup_{l=1}^r A_{n_l}$. 
Thus, $y$ be can written as $\sum_{l=1}^r y^l$, where each $y^l$ is supported in $A_{n_l}$. Let {$M=\max_{1\le l\le r}\|y^l\|_\infty$}. We have then that for every $1\leq l\leq r$, $M B^{n_l}(x)$ is a frequently hypercyclic vector for $\mathfrak B_M(A_{n_l})$, the set of vectors supported in $A_{n_l}$ with coordinate less than $M$. Thus, there are $k_1,\ldots, k_r$ such that for every $1\leq l\leq r$, $\|B^{k_l}(MB^{n_l}(x))-y^l\|<\frac{\varepsilon}{r}.$
It follows that $\|M\sum_{l=1}^r B^{n_l+k_l}(x)-y\|<\varepsilon.$

To prove $(1)$ just observe that $\{P(B)(x):P\in \zK[x]\}$ is a dense subspace formed by frequently recurrent vectors.
On the other hand it is well known that if an operator admits a cyclic and frequently recurrent vector then the $n$-fold product of the operator is frequently recurrent and with a dense lineable set of frequently recurrent vectors, see  \cite[Proposition 3.8]   {CardeMurBlock} or \cite[Proposition 4.2]{GriLopPerQuestions2}.
 \end{proof}

\subsection{Two intermediate conditions\label{intermediate}}
There are two natural cases in which the backward shift operator satisfies conditions  that are formally weaker than frequent hypercyclicity but stronger than frequent recurrence. The first  one is to consider the operators for which the equivalence (2) of Theorem \ref{coro: first Characterization frequent recurrence} holds for $A=\mathbb N_0$.
 Theorem \ref{thm: caract S_1(A)} implies that, in this case, the backward shift operator is frequently hypercyclic for $\mathfrak B_1=\mathfrak B_1(\mathbb N_0)$, the set of vectors with coordinates smaller than one.
\begin{corollary}\label{Teo: caract freq hyp on S1}
    Let $X$ be a Fr\'echet sequence space with  unconditional basis $(e_n)_{n\ge 0}$. The following are equivalent:
\begin{enumerate}
    \item There are sets $A_p$, $\underline {dens}(A_p)>0$ and a sequence $\varepsilon_p\to 0$ such that
    \begin{enumerate}
        \item [i)] for every $p\ge 0,$ $\sum_{n\in A_p} e_{n+p}\in X$ and
        \item [ii)] for every $q\ge 0$ and $m\in A_q$, $\|\sum_{p=0}^\infty \sum_{n\in A_p,n>m}\sum_{j=0}^p e_{n-m+j}\|_q<\varepsilon_q.
    $    
        \end{enumerate}
\item $B$ is frequently hypercyclic for $\mathfrak B_1$.

\end{enumerate}
\end{corollary}

In a similar way to Theorem \ref{cyclic and frequently recurrent vector} we have the following. 
\begin{proposition}
     Let $X$ be a Fr\'echet sequence space. If  $B$ is frequently hypercyclic for $\mathfrak B_1$ then $B$ supports a supercyclic and frequently recurrent vector.
 \end{proposition}
\begin{proof}
If $B$ is frequently hypercyclic for $\mathfrak B_1$, then the proof of Theorem \ref{thm: caract S_1(A)} shows that there is a vector $x\in \mathfrak B_1$, which is frequently hypercyclic for $\mathfrak B_1$. This vector is by definition, frequently recurrent and supercyclic.
\end{proof}

The other case that is worth to consider is when
condition ii) of equivalence (2) in Theorem \ref{Characterization frequent recurrence} is smaller than $\min\{\varepsilon_p,\varepsilon_q\}$. Assuming that $X$ is Banach we obtain an equivalent notion to frequent hypercyclicity.

\begin{proposition}\label{prop: condicion con el min}
    Let $X$ be a Banach space with unconditional basis $(e_n)_{n\ge 0}$. Suppose that there are a subsets $A_p\subseteq \zN$ of positive lower density, a sequence $(N_k)_k$ defined by the relation $N_k=\min_{n>N_{k-1}} \{n\in \bigcup_{p=0}^{k-1}\bigcup_{j=0}^p A_p+N_j\}$ and $\varepsilon_p\to 0$, such that
\begin{enumerate}
    \item [i)]for every $p\in\mathbb N_0$, $\sum_{n\in A_p}e_{n+p}$ converges 
    \item [ii)]
    For every $q\in\mathbb N_0$ and every  $m\in A_q$,
    $$\|\sum_{p=1}^\infty \sum_{n\in A_p, n>m} \sum_{j=0}^p e_{n-m+N_j}\|<\min\{\varepsilon_p,\varepsilon_q\}.$$
   \end{enumerate}
   Then $B$ is frequently hypercyclic.
\end{proposition}
\begin{proof}
By Theorem \ref{caracterizacion hiperciclicidad frecuente} it suffices to prove that there are $A_p$ of positive lower density satisfying i) and the stronger condition
$$\|\sum_{n\in A_p,n>m} e_{n-m+j}\|<\min\{\varepsilon_p,\varepsilon_q\}.$$
We will prove that this holds for a subsequence of $A_p$.

Indeed, considering a subsequence of $A_p$ and using that the basis is unconditional, we may suppose $C\|B\|^{N_p}\varepsilon_p<\frac{1}{2^p}$. Hence, if $p,q\in\zN$ and $m\in A_q$ then
\begin{align*}
  \|\sum_{n\in A_p, n>m} e_{n-m+j}\|&=\|B^{N_p-j}\left(\sum_{n\in A_{p},n>m} e_{n-m+N_p}\right)\|\\
  &\leq C\|B\|^{N_p}\|\sum_{p=1}^\infty \sum_{n\in A_p,n>m}\sum_{j=0}^p e_{n-m+N_j}\|\\
  &\leq C\|B\|^{N_p}\min\{\varepsilon_p,\varepsilon_q\}\leq \min \{\frac{1}{2^p},\frac{1}{2^q}\}.
\end{align*}
\end{proof}

\section{Bilateral backward shifts}\label{section:bilateral}
We now focus  our attention to bilateral backward shifts and extend the results of the preceding section to the bilateral setting. Although all the results in Section \ref{section unilateral} have an extension to the bilateral case, we choose to mention only the main ones.
\begin{theorem}\label{0-1 law frequent recurrence bil}
Let $X$ be a Fr\'echet sequence space with uncondtional basis $(e_n)_{n\in\mathbb Z}$. If $B$ supports a non-trivial frequently recurrent vector then $B$ is frequently recurrent.
\end{theorem}

\begin{proof}
 Let $\varepsilon>0$, $R\in\mathbb Z$, we show that there is a frequently recurrent vector $y\in X$ such that $\|y-e_R\|<\varepsilon$.

Let $x\neq 0$ be a frequently recurrent vector such that $x_R\neq 0$. Let $\delta\in(0,1)$ such that $|x_R|>\delta.$

Since $x$ is a recurrent vector,  there is for every $p\in\mathbb N_0$ a  $q_p>p+|R|$ such that $|x_{q_p}|>\delta$. By the continuity of the coordinate functionals, there are a decreasing sequence of positive numbers $(\varepsilon_{p})_p$ and a sequence of natural numbers $(k_p)_p$  satisfying that for each $p\in\mathbb N_0$, $\varepsilon_{p}<\frac{\varepsilon}{2^{p+1}}$, $k_p>p+1$, and such that  
\begin{align}\label{eq: beta_p (zero one freq) BILAT}
 \begin{split}
  \|w\|_{k_p}<\varepsilon_{p}\Longrightarrow  |w_{j}|<\frac{\delta}{4} ,  \text{ for every } |j|\leq q_p, &\quad \text{and}\\
 \|x-\sum_{|j|\le k_p} x_j e_j\|_{m_p}<\frac{\varepsilon}{2^{p+1}}\frac{\delta}{10C_p},&
 \end{split}
 \end{align}
 where $C_p$ and $\|\cdot\|_{m_p}$ are the unconditional constant basis and seminorm associated to $\|\cdot\|_p$.

For each $p\ge 0$ we consider 
\begin{equation}\label{bilateral eq: A_p zero one freq}
A_p\subseteq N_B(x,U_p) \text{ where } U_p=\{y\in X:\|y-x\|_{k_p+m_p}<\frac{\delta\varepsilon_p}{10C_p}\},
\end{equation}
and the $A_p$ have positive lower density.
Applying the Separation Lemma \ref{separation lemma}
we may assume that the sets $A_p$ are pairwise disjoint and that for every $n\in A_p,\, m\in A_q$ with $n\neq m$ we have that
\begin{equation}\label{bilateral eq separation zero one freq}
|n-m|>k_p+k_q+2|R|>p+1+q+1+2|R|.    
\end{equation}

 By \eqref{eq: beta_p (zero one freq) BILAT} and \eqref{bilateral eq: A_p zero one freq} we have that $|x_{n+q_p}-x_{q_p}|<\frac{\delta}{4}$ and hence $|x_{n+q_p}|>\frac{\delta}{2}$ for each $n\in A_p.$ By unconditionality this implies that $\sum_{n\in A_p} e_{n+q_p}$ converges and by applying the backward shift $q_p-p$ times it follows that $\sum_{n\in A_p} e_{n+p+1}$ converges. By the finite invariance property of  the sets of positive lower density, we may assume that $A_p\sub [k_p+|R|+1,\infty)$ and that,  for every $|j|\le p+1,$
\begin{equation}\label{bilateral eq: sumas de e_j norma chica (zero one freq)}
 \|\sum_{n\in A_p} e_{n+j}\|<\frac{\varepsilon}{(p+1)2^{p+1}}.   
\end{equation}

Applying Proposition \ref{prop: condic necesaria vectores recurrentes} and Remark \ref{rem:M_k>N_k} there is a sequence $(M_p)_p$ for which 
\begin{equation}\label{eq: bilateral rat en U_0}
    RAT((A_{M_p})_p)\subseteq N_B(x,U_0)
\end{equation} and such that $M_p>N_p$ for every $p\geq 1$. The fact that $M_p>N_p$ together with the contention $A_p\sub [k_p+|R|+1,\infty)$, implies that $(A_{M_p})_p$ is compatible, i.e. 
 $RAT((A_{M_p})_p)=\{N_k\}$.

Let $$y=\sum_{n\in RAT((A_{M_p})_p)} e_{n+R}=e_R+\sum_{p=0}^{\infty}\sum_{n\in A_{M_p}}\sum_{j=0}^pe_{n+N_j+R},$$
which we claim to be frequently recurrent.

Using \eqref{bilateral eq: sumas de e_j norma chica (zero one freq)} and applying the same arguments as in \eqref{eq: def y^{k+1}} it can be shown that $y\in X$ and that $\|y-e_R\|<\varepsilon$.

We will show now that $y$ is frequently recurrent. 
 Fix $q\in\zN$ and $m\in A_{M_{q}}$.
 
Noting that $m$ belongs only to $A_{M_{q}}$, because the sets $A_{M_p}$ are pairwise disjoint, we get that
\begin{align*}
B^m(y)-\sum_{j=-|R|}^{N_{q}+|R|}y_je_j &=B^m(y)-\sum_{j=0}^{q}e_{N_j+R}\\
&=e_{R-m}+\sum_{p\neq q}\sum_{n\in A_{M_p}}\sum_{j=0}^{p}e_{n-m+N_j+R}+\sum_{n\in A_{M_{q}}}\sum_{j=0}^{q}e_{n-m+N_j+R}-\sum_{j=0}^{q}e_{N_j+R}\\
&=e_{R-m}+\sum_{p=0}^\infty \sum_{n\in A_{M_p},n\neq m}\sum_{j=0}^pe_{n-m+N_j+R}.
\end{align*}

By \eqref{eq: bilateral rat en U_0} we have that $RAT((A_{M_p})_p)\subseteq N_B(x,U_0)$. Thus, we have that for every $n\in A_{M_p}$ and every $0\le j\le p,$ that
$$
B^{n+N_j}x\in U_0,\quad\text{and in particular,}\quad \| B^{n+N_j}x -x\|_{k_0+m_0}<\varepsilon_0.
$$
Then since $|x_R|>\delta,$ we have by \eqref{eq: beta_p (zero one freq) BILAT} that 
\begin{equation*}
    |x_{n+N_j+R}|\ge |x_R|-|x_{n+N_j+R}-x_R| \ge \delta - \frac{\delta}{4}>\frac{\delta}{4}.
\end{equation*}

 Moreover, since by \eqref{bilateral eq separation zero one freq}, $|n-m|>N_p+k_{N_{q}}+|R|$ whenever $n\in A_{M_p}$,  it follows that the $|n-m+N_j+R|>k_{N_{q}}$ for each $0\le j \le p$. Note also that since $\min A_{M_q}>k_{M_q}+|R|$ we have $R-m<-k_{M_q}\leq- k_{N_q}$, and thus,
  \begin{align*}
 \begin{split}
  \|B^m(y)-\sum_{j=-|R|}^{N_{q}+|R|}y_je_j \|_{q}&=\left\|e_{R-m}+\sum_{p=0}^\infty \sum_{n\in A_{M_p},n\neq m}\sum_{j=0}^pe_{n-m+N_j+R}\right\|_{q}\\
&= \left\|\frac1{x_{R}}x_{R}e_{R-m}+\sum_{p=0}^{\infty}\sum_{n\in A_{M_p},n\neq m}\sum_{j=0}^p\frac1{x_{n+N_j+R}}x_{n+N_j+R}e_{n-m+N_j+R}\right\|_{q}\\
 &\le \left\|\frac1{x_{R}}[B^m(x)-\sum_{j=-k_{N_{q}}}^{k_{N_{q}}}x_je_j]_{R-m}e_{R-m}\right\|_q\\
 &+\left\|\sum_{p=0}^{\infty}\sum_{n\in A_{M_p},n\neq m}\sum_{j=0}^p\frac1{x_{n+N_j+R}}[B^m(x)-\sum_{j=-k_{N_{q}}}^{k_{N_{q}}}x_je_j]_{n-m+N_j+R}e_{n-m+N_j+R}\right\|_q\\
 &\leq \Big(\frac{C_q}{\delta}+\frac{4C_q}{\delta}\Big)\left(\|B^m(x)-x\|_{m_q}+\|x-\sum_{j=-k_{N_{q}}}^{k_{N_{q}}}x_je_j\|_{m_q}\right)< \frac{\varepsilon}{2^{{q}}},
  \end{split}
 \end{align*}
where the last inequality follows from \eqref{eq: beta_p (zero one freq) BILAT} and \eqref{bilateral eq: A_p zero one freq}.
 This implies that $y$ is a frequently recurrent vector.
 
\end{proof}

In the same way as in the unilateral case we can deduce a characterization for  bilateral frequently recurrent backward shifts.
 
\begin{theorem}\label{Characterization frequent recurrence bil}
Let $X$ be a Banach space with unconditional basis $(e_n)_{n\in\mathbb Z}$. The following are equivalent: 
\begin{enumerate}
\item $B$ is frequently recurrent and
\item
 there are a sequence $\varepsilon_p\to 0$ and sets of positive lower density $A_p$ such that
\begin{enumerate}
    \item [i)]for every $p\in\mathbb N_0$, $\sum_{n\in A_p}e_{n+p}$ converges 
\end{enumerate}
and such that one of the following conditions holds
\begin{enumerate}
    
    \item [ii)]For every $q\in \zN$ and every $m\in A_q$ we have that
    $$\Big\|\sum_{n\in RAT((A_p)), n\neq m} e_{n-m}\Big\|_q<\varepsilon_q;$$

\end{enumerate}
\item There is a compatible sequence $(A_p)_p$ formed by sets of positive lower density such that
$$\sum_{n\in RAT((A_p)_p)} e_n \text{ is a frequently recurrent vector}.$$
\end{enumerate}

\end{theorem}
\begin{remark}\label{rem: condition ii' bil}
Condition $ii)$ can also  be replaced by:
    \begin{itemize}
    \item [ii')]
     If $N_j$ is the sequence defined by $N_k=\min_{n>{N_{k-1}}} \bigcup_{p=0}^{k-1} \bigcup_{j=0}^p A_p+N_j$ and $N_0=0$, then 
    for every $q\in\mathbb N_0$ and every  $m\in A_q$ we have that
    $$\|\sum_{p=0}^\infty \sum_{n\in A_p, n\neq m} \sum_{j=0}^p e_{n-m+N_j}\|_q<\varepsilon_q$$
    
    \end{itemize}
    
\end{remark}

\begin{proof}[Sketch of the proof of Theorem \ref{Characterization frequent recurrence bil}.]
As in the unilateral case, the implications $(1)\Rightarrow (2)$ and $(1)\Rightarrow (3)$ can be extracted from the proof of  Theorem \ref{0-1 law frequent recurrence bil}. By the zero-one law Theorem for bilateral shifts we have that $(3)\Rightarrow (1).$

(2)$\Rightarrow$ (1).  
 By Theorem \ref{0-1 law frequent recurrence bil} it suffices to construct a nonzero frequently recurrent vector.
Let $(N_j)_j$ be the sequence associated to $RAT((A_p)_p)$ (Or just $N_j$ if we are assuming ii')). 
 
 Applying the Separation Lemma \ref{separation lemma} and finite invariance there are subsets $A_p'\subseteq A_p$ such  that,  for every $p,q\in\zN$, every $n\in A_p'$ and $m\in A_q'$, and every $0\leq j\leq p$, 
\begin{equation}\label{separation property characterization bil}
|n-m|>p+q,\qquad \text{
 $A_p'\subseteq [N_p+1,\infty)$, }\qquad \|\sum_{n\in A_p'} e_{n+j}\|\leq \frac{1}{2^p}.    
\end{equation}

We apply now Lemma \ref{lem: RAT subconjuntos} to obtain a sequence $(M_p)_p$ such that $(A_{M_p}')_p$ is compatible and $RAT((A_{M_p}')_p)\sub RAT((A_p)_p)$.
 Moreover, if $RAT((A_{M_p}')_p)=(N_j')_j$, $N_{l}'=\min_{n>N_{l-1}'} \bigcup_{p=0}^{l-1} \bigcup_{j=0}^p A_{M_p}'+N_j'$ for every $l$, we may assume that 
\begin{enumerate}
   \item[a)] $M_{l}>\max\{m_{l},N_l'\}$ for every $1\leq l$, (see Remark \ref{rem: cosas del lemma que usamos})  
   \item [b)] $\varepsilon_{M_l}<\varepsilon_{M_{l-1}}$ for $1\leq l$, 
    \item [c)] $C_l\cdot\varepsilon_{M_l}<\frac{1}{l}$ for every $1\leq l$,
    \item [d)] $\{N_1',\ldots N_k'\}\subseteq \{N_1,\ldots N_k\}$,
\end{enumerate}
where $C_l,m_l$ are the constants from the unconditionality.

Either if ii) or ii') holds, we have by unconditionality that
\begin{align}
\nonumber \left\|\sum_{\overset{N\in RAT((A_{M_p}')_p)}{N\neq m}} e_{N-m}\right\|_q & =    \left\|\sum_{p=0}^\infty\sum_{\overset{n\in A_{M_{p}}'}{ n\neq m}}\sum_{j=0}^p e_{n-m+N_j'}\right\|_q \leq C_q\cdot \left\|\sum_{p=0}^\infty\sum_{\overset{n\in A_{M_p}}{n\neq m}}\sum_{j=0}^{M_p} e_{n-m+N_j}\right\|_{m_q}\\
 \label{eq: norma de la suma n>m <1/q BIL}   & \leq C_q\cdot\left\|\sum_{p=0}^\infty\sum_{\overset{n\in A_{M_p}}{n\neq m}}\sum_{j=0}^{M_p} e_{n-m+N_j}\right\|_{M_q}
     <C_q\cdot \varepsilon_{M_{q}}
    <\frac{1}{q}. 
\end{align}
Let $y=\sum_{p=0}^\infty \sum_{n\in A_{M_p}'}\sum_{j=0}^p e_{n+N_j'}.$ This vector is well-defined by \eqref{separation property characterization bil} and it is frequently recurrent because if $m\in A'_{M_q}$ then
 $$B^m(y)-\sum_{j=0}^{N_q'} y_j=B^m(y)-\sum_{j=0}^q e_{N_{j}'}= \sum_{p=0}^\infty \sum_{n\in A_{M_p'}, n\neq m} \sum_{j=0}^p e_{n-m+N_j'}$$
 and,   by \eqref{eq: norma de la suma n>m <1/q BIL},
 $$
 \Bigg\|\sum_{p=0}^\infty \sum_{n\in A_{M_p'}, n\neq m} \sum_{j=0}^p e_{n-m+N_j'}\Bigg\|_q<\frac{1}{q}.
 $$
\end{proof}

Proceeding as in Theorems  \ref{thm: caract S_1(A)} and \ref{cyclic and frequently recurrent vector}, but using now the zero-one law from Theorem \ref{0-1 law frequent recurrence bil} and the characterization for bilateral backward shifts in Theorem \ref{Characterization frequent recurrence bil}, we can also prove the following results.   
\begin{theorem} \label{thm: caract S_1(A) bilat}
Let $X$ be a Fr\'echet sequence space with unconditional basis $(e_n)_{n\in \mathbb Z}$. Then $B$ is frequently recurrent if and only if there are $A\subseteq \mathbb N_0$ with $\underline {dens}(A)>0$ and $x\in X$ supported in $A$ such that $x$ is a frequently hypercyclic vector for $\mathfrak B_1(A)$. 
\end{theorem}

\begin{theorem}\label{cyclic and frequently recurrent vector bil}
Let $X$ be a Fr\'echet sequence space with unconditional basis $(e_n)_{n\in\mathbb Z}$. Then  $B$ is frequently recurrent if and only if there exists a cyclic and frequently recurrent vector.

In particular, if $B$ is frequently recurrent then
\begin{enumerate}
    \item the set of frequently recurrent vectors is dense lineable and moreover,
    \item the set of frequently recurrent vectors of $B\times\cdots \times B:X^n\to X^n$ is dense lineable for every $n\in\mathbb N.$
\end{enumerate}
\end{theorem}

\section{Final remarks}

An important open question, posed in \cite{BonGroLopPer22JFA}, is whether every frequently recurrent backward shift is indeed frequently hypercyclic. For boundedly complete bases, there is an easy positive answer: by \cite{charpentier2019chaos}, every reiteratively hypercyclic backward shift operator over a boundedly complete basis is frequently hypercyclic, and by \cite{CardeMurBlock} every reiteratively recurrent backward shift operator is reiteratively hypercyclic. However, the proof breaks down for more general spaces. 

Inspired by the results of Subsection \ref{intermediate} we can split this problem into two different questions.

\textbf{Question 1} Let $B$ be frequently hypercyclic for the set of coordinates bounded by $1$. Is $B$ frequently hypercyclic?

\textbf{Question 2} Let $B$ be a frequently recurrent backward shift operator. Is $B$ frequently hypercyclic for the set of coordinates bounded by $1$?

By \cite{GriLop23}, the existence of invariant measures is related to the existence of frequently recurrent vectors in the support of the measure.
It is thus natural to ask if a zero-one type law
holds for invariant measures of weighted shifts: does the existence of a non-trivial invariant measure imply the existence of an invariant measure of full support? 

If $B$ is a backwardshift operator acting on a Banach space $X$ in which $\{e_n\}$ is a boundedly complete and unconditional basis (in this case $B$ is an adjoint operator) then using the zero-one law we can give a positive answer and moreover
$B$ supports a non-trivial invariant measure  if and only if it has an ergodic measure of full support.
This also provides an alternative proof of  \cite[Proposition 2.1]{grivaux2024orthogonality}.

Indeed, according to \cite[Lemma 3.1]{GriLop23} the existence of non-trivial invariant measure implies the existence of a nonzero frequently recurrent vector. By the zero-one law Theorem \ref{0-1 law frequent recurrence} this implies that the operator is frequently recurrent and since $B$ is an adjoint operator then the operator is frequently hypercyclic and chaotic \cite[Theorem 2.1]{charpentier2019chaos}. In particular, it satisfies the frequent hypercyclicity criterion \cite{GroPer11}[See Proposition 9.13 and Theorem 4.8], which in turn implies ergodicity \cite[Theorem 1]{MurPer13}.

For non-adjoint operators or in the case of non-normable Fr\'echet spaces, the situation is more subtle. 
We thus pose the following question.

\textbf{Question 3.} Let $B$ be a backwardshift operator on a Fr\'echet (or Banach) sequence space $X$. Suppose that $B$ has a non-trivial invariant measure. Does $B$ have an invariant measure of full support?

We remark that for an adjoint operator operator on a Fr\'echet space $X$ such that its predual is a quasi-$\ell_\infty$-barreled Hausdorff locally convex topological vector space, sufficient conditions for the existence invariant measures of full support were proved in \cite[Theorem 2.3]{Lop24Locallybounded}, which are related to the existence of frequently recurrent vectors with locally bounded orbit.

Finally we note that the results presented here can be extended to $\mathcal F$-recurrence for families $\mathcal F$ which are finitely invariant and satisfy a separation lemma (like Lemma \ref{separation lemma}, see also \cite{MarMenPui22,BayGri04}). But, for many families, like upper or block families, the results are simpler to prove (see \cite{bonillazerofurstenberg_preprint}).

\subsection*{Acknowledgments} We would like to express our gratitude towards the reviewer, who read carefully the entire manuscript and beside other helpful comments, pointed out a significant gap in the previous version. Their comments motivated us to analyze in detail the arithmetic structure of the return time sets.

 \bibliographystyle{abbrv}

\end{document}